\documentclass[a4paper,11pt]{amsart}
\usepackage{latexsym}
\usepackage{amssymb}
\usepackage{color}
\usepackage[
frame,cmtip,arrow,matrix,line,graph]{xy}

\let\oldtocsection=\tocsection

\let\oldtocsubsection=\tocsubsection

\let\oldtocsubsubsection=\tocsubsubsection

\renewcommand{\tocsection}[2]{\hspace{0em}\oldtocsection{#1}{#2}}
\renewcommand{\tocsubsection}[2]{\hspace{1em}\oldtocsubsection{#1}{#2}}
\renewcommand{\tocsubsubsection}[2]{\hspace{2em}\oldtocsubsubsection{#1}{#2}}

\definecolor{darkgreen}{RGB}{47,139,79}
\definecolor{darkblue}{RGB}{36,24,130}

\usepackage{hyperref}
\hypersetup{
    colorlinks=true, 
    linktoc=all,     
    linktocpage,
    citecolor=darkgreen,
    filecolor=black,
    linkcolor=darkblue,
    urlcolor=black
}

\input xy
\xyoption{all}

\usepackage{amsmath}
\usepackage{amsthm}
\usepackage{amsfonts}
\usepackage{amscd}

\usepackage{epsfig}
\usepackage{lpic,enumerate}

\newtheorem{thm}{Theorem}[section]
\newtheorem{lem}[thm]{Lemma}
\newtheorem{prop}[thm]{Proposition}
\newtheorem{rem}[thm]{Remark}
\newtheorem{ex}[thm]{Example}
\newtheorem{cor}[thm]{Corollary}
\newtheorem{Def}[thm]{Definition}

\newtheorem{Th}{Theorem}

\newcommand{\A}{\mathcal{A}}
\newcommand{\Ann}{\mathcal{A}nn}
\newcommand{\Ai}{\mathcal{A}_\infty}
\newcommand{\Ass}{\mathcal{A}ss}

\newcommand{\C}{\mathcal{C}}

\newcommand{\bC}{\overline{C}}

\newcommand{\bD}{\overline{D}}

\newcommand{\e}{\mathcal{E}}

\newcommand{\uj}{\underline{j}}
\newcommand{\uk}{\underline{k}}
\newcommand{\LL}{\mathcal{L}}

\newcommand{\N}{\mathbb{N}}

\newcommand{\pp}{\mathcal{P}}
\newcommand{\OO}{\mathcal{O}}
\newcommand{\OC}{\mathcal{OC}}

\newcommand{\QQ}{\mathbb{Q}}
\newcommand{\RR}{\mathbb{R}}

\newcommand{\SD}{\mathcal{SD}}

\newcommand{\Z}{\mathbb{Z}}

\newcommand{\Hom}{\textrm{Hom}}

\newcommand{\Comp}{\operatorname{Ch}}

\newcommand{\al}{\alpha}
\newcommand{\be}{\beta}

\newcommand{\ga}{\gamma}
\newcommand{\Ga}{\Gamma}

\newcommand{\De}{\Delta}

\newcommand{\la}{\lambda}

\newcommand{\s}{\sigma}
\newcommand{\Si}{\Sigma}

\newcommand{\surj}{\twoheadrightarrow}

\newcommand{\rar}{\longrightarrow}
\newcommand{\inc}{\hookrightarrow}
\newcommand{\sta}{\stackrel}
\newcommand{\arsim}{\sta{\simeq}{\rar}}
\newcommand{\minus}{\backslash}

\newcommand{\x}{\times}

\newcommand{\ot}{\otimes}

\newcommand{\w}{\wedge}

\newcommand{\lgl}{\langle}
\newcommand{\rgl}{\rangle}

\newcommand{\del}{\partial}

\newcommand{\oc}[2]{[\begin{subarray}{c} #2 \\ #1 \end{subarray}]}

\newcommand{\CPart}{\operatorname{K}}
\newcommand{\Nat}{\operatorname{Nat}}
\newcommand{\bNat}{\overline{\operatorname{Nat}}}
\newcommand{\Fun}{\operatorname{Fun}}
\newcommand{\Inj}{\operatorname{Inj}}
\newcommand{\End}{\operatorname{End}}

\newcounter{samcounter}

\renewcommand{\u}[1]{\underline{#1}}

\begin{document}

\bibliographystyle{plain}

\title[Universal operations in  Hochschild homology]{Universal operations in  Hochschild homology}

\author{Nathalie Wahl}
\email{wahl@math.ku.dk}

\date{\today}

\begin{abstract}
We provide a general method for finding all natural operations on the
Hochschild complex of $\e$-algebras, where $\e$ is any algebraic structure encoded in a prop with
multiplication, as for example the prop of Frobenius, commutative or
$A_\infty$-algebras. We show that the chain complex of all such natural operations is approximated by  a certain chain complex of {\em formal operations}, for
which we provide an explicit model that we can calculate in a number of cases.
When $\e$ encodes the structure of open topological conformal field
theories, we identify this last chain complex, up quasi-isomorphism,  with the moduli space of Riemann surfaces with boundaries,
thus establishing that the operations constructed by Costello and 
Kontsevich-Soibelman via different methods identify with all formal operations.
When $\e$ encodes open topological quantum field theories (or symmetric Frobenius algebras) our chain complex identifies with Sullivan diagrams, 
thus showing that operations constructed by Tradler-Zeinalian, again by different methods, account for all formal operations. As an illustration of the last result we
exhibit two infinite families of non-trivial operations  and use these to produce non-trivial higher string topology operations, which had so far been elusive. 

\end{abstract}

\maketitle

\section*{Introduction}

To any dg-algebra, or more generally $\Ai$--algebra $A$, one can associate the Hochschild chain complex $C_*(A,A)$ of the algebra with
coefficients in itself (as in e.g.~\cite[7.2.4]{KS06}).  Many authors have constructed operations on this complex, i.e.~natural maps 
$$\nu_A:C_*(A,A)^{\ot p}\to C_*(A,A)^{\ot q}$$
for some $p,q\ge 0$, under various assumptions on the type of $A$, see e.g.~\cite[Chap 4]{LodCyclic},\cite{NesTsy99} or \cite{KS06}.

In the present paper, we systematically address the question of finding the chain complex $\Nat_\e^\ot(p,q)$ of {\em all} such natural operations, for various classes of $\Ai$--algebras with extra structure $\e$  (associative, commutative, Frobenius,
\ldots).  We do this by introducing a chain complex $\Nat_\e(p,q)$ of all {\em formal operations}, loosely speaking obtained by forgetting a symmetric monoidal
condition in the definition of the algebras, and provide an explicit model for the chain complex of all formal operations
(Theorem~\ref{natintro}). We are able to
identify this chain complex up to quasi-isomorphism in terms of well-known objects, in several interesting cases (see Theorems~\ref{OCintro} and \ref{SDintro}). The chain complex of
formal operations comes with a chain map $r: \Nat_\e(p,q) \to  \Nat^\ot_\e(p,q)$ to the natural operations. While this map is not an isomorphism in
general, we do show that it has good properties, 
under various assumptions on $\e$, e.g., if $\e$ comes from an operad. And for a general $\e$ we identify the natural operations in terms of the formal operations of the
completion $\widehat \e$ of $\e$ via an isomorphism $\Nat_{\e}^\ot(p,q)  =  \Nat_{\widehat \e}^\ot(p,q) \xleftarrow{\cong} \Nat_{\widehat \e}(p,q)$.
Before detailing our computations, 
we start by making precise which types of algebraic structures we consider.

\medskip

An $\Ai$--algebra, as originally defined by Stasheff \cite{Sta63}, can be described as an enriched symmetric monoidal functor 
$\Phi:\Ai\rar\Comp$
from a certain dg-category  $\Ai$ with objects the natural numbers, to $\Comp$, the category of chain
complexes over $\Z$: the $\Ai$--algebra associated to such a functor $\Phi$ is its value at 1, together with  multiplications $m_i$, $i\ge 2$, 
encoded as morphisms in the category $\Ai$. 

Let $\e$ be any symmetric monoidal dg-category with objects the natural numbers
(i.e.~a prop, also denoted PROP  \cite[\S 24]{McL65}), equipped with a symmetric monoidal dg-functor $i:\Ai\to\e$ which is the identity on objects---we call such a pair $(\e,i)$ a {\em prop with $\Ai$--multiplication}. 
An $\e$-algebra is defined to be a strong symmetric monoidal functor $\Phi:\e\to\Comp$, i.e.~a functor
equipped with natural isomorphisms
$\Phi(1)^{\ot n}\sta{\cong}{\to} \Phi(n)$ compatible with the symmetries of $\e$. 
If $(\e,i)$ is a prop with $\Ai$--multiplication, $i$ endows $\Phi(1)$ with the structure of an $\A_\infty$--algebra. 
We define the Hochschild complex of the $\e$--algebra $\Phi$ with respect to $i$ as the standard Hochschild complex of the $\Ai$--algebra $\Phi(1)$, i.e.
$${\bf C}_*(\Phi):=C_*\big(\Phi(1),\Phi(1)\big)=\bigoplus_{n\ge 1}\Phi(1)^{\ot n}$$
with the standard Hochschild differential defined using the multiplication and higher multiplications of $\Phi(1)$.


A natural operation, with $p$ inputs and $q$ outputs, on the Hochschild complex of $\e$--algebras is the data of a linear map
$$\nu_\Phi:{\bf C}_*(\Phi)^{\ot p}\to {\bf C}_*(\Phi)^{\ot q}$$
for any $\e$--algebra $\Phi$, which is natural with respect to maps of $\e$--algebras. 
The set of such linear maps between chain complexes is again canonically a chain complex.
In other words, if we let $\Fun^\ot(\e,\Comp)$ denote the category of $\e$--algebras  and\\  ${C}_\e^{\ot p}:\Fun^\ot(\e,\Comp)\rar
\Comp$  the functor taking an $\e$--algebra $\Phi$ to the $p$th power of its Hochschild complex ${\bf C}_*(\Phi)^{\ot p}$, 
the natural  operations, with $p$ inputs and $q$ outputs,  on the Hochschild complex of
$\e$--algebras is the chain complex $\Hom({C}_\e^{\ot p}, {C}_\e^{\ot q})$.

More generally, we want to consider operations of the form 
$$\nu_\Phi:{\bf C}_*(\Phi)^{\ot n_1}\ot \Phi(1)^{\ot m_1}\to {\bf C}_*(\Phi)^{\ot n_2}\ot \Phi(1)^{\ot m_2}$$
allowing also copies of the algebra $\Phi(1)$ itself as input and output. Generalising the previous paragraph, the chain complex of those 
$(\oc{m_1}{n_1},\oc{m_2}{n_2})$--operations is
hence given by
$$\Nat_\e^\ot(\oc{m_1}{n_1},\oc{m_2}{n_2}) := \Hom ({C}_\e^{\ot (n_1,m_1)}, {C}_\e^{\ot (n_2,m_2)})$$ 
where  ${C}_\e^{\ot (n,m)}:\Fun^\ot(\e,\Comp)\rar \Comp$  is the functor  taking $\Phi$ to ${\bf
  C}_*(\Phi)^{\ot n}\ot \Phi(1)^{\ot m}$.  

One can in fact define a Hochschild complex ${\bf C}_*(\Phi)$ for {\em any} functor $\Phi: \e \to \Comp$ (not necessarily strong symmetric monoidal) by
setting $${\bf C}_*(\Phi):=\bigoplus_{n\ge 1}\Phi(n),$$
noting that the usual Hochschild differential is still well-defined.
Letting\\ ${C}_\e^{(n,m)}: \Fun(\e,\Comp)\rar \Comp$ more generally denote the analogous
extension of $C_\e^{\ot(n,m)}$, with value on $\Phi$ defined as a sum of terms $\Phi(k_1+\dots+k_n+m)$, we can hence consider the chain complex of 
all the natural $(\oc{m_1}{n_1},\oc{m_2}{n_2})$--operations
 $$\Nat_\e(\oc{m_1}{n_1},\oc{m_2}{n_2}) := \Hom(C^{(n_1,m_1)}_\e, C^{(n_2,m_2)}_\e)$$ on the Hochschild complex of generalized $\e$-algebras.



There is  a restriction map
$$r: \Nat_\e(\oc{m_1}{n_1},\oc{m_2}{n_2})\rar \Nat_\e^\ot(\oc{m_1}{n_1},\oc{m_2}{n_2})$$
which in general need neither be injective, surjective, nor a quasi-isomorphism. We think of the left-hand side as {\em formal
  operations} on $\e$-algebras, and refer to them as such.
We show that the prop $\e$ can always be replaced by a new prop $\widehat\e$, its {\em completion},  with same category of algebras and satisfying 
$$\Nat^\ot_{\e}(\oc{m_1}{n_1},\oc{m_2}{n_2})\cong\Nat^\ot_{\widehat\e}(\oc{m_1}{n_1},\oc{m_2}{n_2})\cong\Nat_{\widehat\e}(\oc{m_1}{n_1},\oc{m_2}{n_2}),$$
so natural operations can always be described as formal operations for a different prop 
(see Definition~\ref{compdef} and Corollary \ref{compcor})---the completion $\widehat\e$ is in fact the prop whose morphisms are $\widehat \e(m_1,m_2)= \Nat_\e^\ot(\oc{m_1}{0},\oc{m_2}{0})$. 

We show that the restriction map $r$ is injective (resp.~surjective or a quasi-isomorphism) for all  $(\oc{m_1}{n_1},\oc{m_2}{n_2})$ if and only if it is injective
(resp.~surjective or a quasi-isomorphism)
for all  $(\oc{m_1}{0},\oc{m_2}{0})$; or, said differently, it is injective if and only if any non-zero morphism $f \in \e(m_1,m_2)$ acts non-trivially on some $\e$-algebra,
it is surjective if any operation on $\e$--algebras is induced by $\e$, and it is a quasi-isomorphism if and only if $\e \to \widehat
\e$ is a quasi-isomorphism (see Theorem~\ref{compthm} for the first two statements and Corollary~\ref{Nathtpy} for the last one).
The map $r$ is for example always injective if $\e$ is the prop associated to an operad, and one can show that it is an isomorphism when restricting to 
operads and algebras in vector spaces,  over a field of characteristic 0 \cite{Fre04} (see Examples~\ref{opex1} and \ref{opex2} for more details).

Our main technical theorem, Theorem~\ref{natural}, gives an explicit description of the chain complex 
$\Nat_\e(\oc{m_1}{n_1},\oc{m_2}{n_2})$: 

\begin{Th}[see Theorem~\ref{natural}]\label{natintro} For any prop with $\Ai$--multiplication $(\e,i)$, there is an isomorphism of  chain complexes
$$\Nat_\e(\oc{m_1}{n_1},\oc{m_2}{n_2}))\cong \prod_{j_1,\dots,j_{n_1}\ge 1}\ \bigoplus_{k_1,\dots,k_{n_2}\ge 1}\!\! \e(j+m_1,k+m_2)[k-j+n_1-n_2],$$ 
where $j=j_1+\dots+j_{n_1}$, $k=k_1+\dots+k_{n_2}$, and where the differential on the right hand
side is the sum of the differential of $\e$, a multi-Hochschild and multi-coHochschild differential. The square brackets indicate a shift in grading.
\end{Th}


Theorem~\ref{natural} gives in addition a version for the reduced Hochschild complex when $\e$ has units. 
A version for higher Hochschild homology (in the sense of \cite{Pir00}) can be found in \cite{Angelathesis}. 

Using a spectral sequence for the right hand side, 
a corollary of the theorem is that $\Nat_\e$ is homotopy invariant in
$\e$ (see Corollary~\ref{Nathtpy}).

\medskip

We 
use the same spectral sequence to identify the homology of $\Nat_\e(\oc{m_1}{n_1},\oc{m_2}{n_2})$ in three cases detailed below: the case of unital 
$\Ai$--algebras, the case of Frobenius algebras, and the case of open field theories. Klamt computed in addition the case of commutative algebras in
\cite{Kla13A}, and, to a large extend, the case of commutative Frobenius algebras in \cite{Kla13B}. To further exemplify our approach, we show in Proposition~\ref{capprop} how the cap product in Hochschild homology can be seen as
part of the chain complex of formal operations $\Nat_\e(\oc{0}{1},\oc{0}{1})$ for $\e=\End(A)$ the endomorphism prop of the algebra $A$ considered. 

\medskip

Now, consider the case where $\e = \OO$, the open cobordism category. It is a  prop with morphisms $\OO(m_1,m_2)$ a
chain model given in terms of fat graphs for the moduli space of Riemann cobordisms from $m_1$ closed unit intervals to $m_2$ 
closed unit intervals. There is a  natural inclusion $i:\Ai\to\OO$ induced by including forests into all graphs, so $\OO$ is a prop with $\Ai$--multiplication. Furthermore, let
$\OC(\oc{m_1}{n_1},\oc{m_2}{n_2})$ denote a chain model for the moduli space of Riemann cobordisms from $n_1$ circles and $m_1$ intervals to $n_2$
circles and $m_2$ intervals, where each component of the cobordism has at
least one incoming or one free (i.e.~neither incoming nor outgoing) boundary component; see Section~\ref{OCsec}
for precise definitions. Theorem~\ref{OC} provides the following description of the formal operations on the Hochschild complex of open field
theories: 

\begin{Th}[see Theorem~\ref{OC}]\label{OCintro} There is an inclusion
$$\OC(\oc{m_1}{n_1},\oc{m_2}{n_2} )\stackrel\sim\inc \Nat_{\OO}(\oc{m_1}{n_1},\oc{m_2}{n_2})$$
which is a split-injective quasi-isomorphism.
\end{Th}
That the open-closed cobordism category $\OC$ can be used to construct natural operations on the Hochschild
complex of $\OO$--algebras was first discovered by Costello and Kontsevich--Soibelman via 
a priori different constructions \cite{costello07,KS06}. What we
show here is that these, up to quasi-isomorphism, both give all {\em formal} operations. 
This gives an answer to a wish expressed in \cite[1.3]{costello07} of a natural  algebraic characterization of the category of chains on 
moduli spaces of curves in terms of the functors which assign to an $\OO$--algebra the tensor powers of its Hochschild chains.

We also obtain a computation of the formal operations in the
case  of unital $\Ai$--algebras, i.e.\ $\e=\Ai^+$, from the computation of $\OO$, using that the prop of unital $\Ai$--algebras $\Ai^+$ is a subcategory of $\OO$. In fact $\Nat_{\Ai^+}(\oc{m_1}{n_1},\oc{m_2}{n_2})$ is a union of components of $\Nat_\OO(\oc{m_1}{n_1},\oc{m_2}{n_2})$, namely those
 associated to cobordisms which are a union of certain annuli and discs, generated up to the above quasi-isomorphism by the well-known operations, namely the $\Ai$--structure of the
 algebra, the inclusion of the algebra in its Hochschild complex, and Connes-Rinehart's boundary operator $B$ (see Theorem~\ref{Ai}). 

\medskip

Finally we consider the case 
$\e=H_0(\OO)$, the 0th homology of $\OO$. Symmetric monoidal functors $H_0(\OO)\to \Comp$ are also known as open topological quantum field theories,
which by \cite[Cor 4.5]{LauPfe} correspond precisely to symmetric Frobenius algebras. 
In \cite{Tradler-Z}, Tradler and Zeinalian show that a certain chain complex of Sullivan diagrams acts on the Hochschild cochain complex of
symmetric Frobenius algebras. (See also \cite{WahWes08} where the dual action on the Hochschild chains is constructed.) 
Our Theorem \ref{H0OC} shows that these define all formal operations in the following sense:

\begin{Th}[see Theorem~\ref{H0OC}]\label{SDintro} 
There is an inclusion
$$\SD(\oc{m_1}{n_1},\oc{m_2}{n_2}) \stackrel\sim\inc \Nat_{H_0(\OO))}(\oc{m_1}{n_1},\oc{m_2}{n_2}) $$
which is a split-injective quasi-isomorphism
where  $\SD(\oc{m_1}{n_1},\oc{m_2}{n_2})$  is a chain complex of Sullivan diagrams on $n_2$ circles, with $n_1$ ``incoming'' boundary
cycles and $m_1+m_2$ labeled leaves. 
\end{Th}
The chain complex $\SD(\oc{m_1}{n_1},\oc{m_2}{n_2})$ is a quotient of $\OC(\oc{m_1}{n_1},\oc{m_2}{n_2})$
with the same $H_0$, and is briefly recalled in Section~\ref{SDsec}. It computes the homology of the harmonic compactification of moduli space
\cite{EgaKup}. 

Beyond providing a model for the right-hand side, Theorem~\ref{SDintro} can also be used ``in reverse'' to give new information about the left-hand side:  
The theorem indeed implies that the Hochschild homology of Frobenius algebras can be used as representation of the homology of Sullivan diagrams, and  
in Section~\ref{non-trivial} we use this representation for the Frobenius algebra $H^*(S^n)$ to produce two infinite families of non-trivial homology classes in Sullivan diagrams of increasing degree   
(non-trivial both with integral and rational coefficient). 
The above cycles furthermore give rise to rational higher string operations, since for any 1--connected
manifold $M$ there is an $H_0(\OO,\QQ)$--algebra $A(M)$ whose Hochschild homology is isomorphic to $H^*(LM,\QQ)$ (using \cite{lambrechts_stanley}
and \cite{jones}, see also \cite{felix_thomas}  or \cite[Sec 6.6]{WahWes08}), and these operations  are non-trivial since they are non-trivial on
$HH_*(H^*(S^n,\QQ),H^*(S^n,\QQ))\cong H^*(LS^n,\QQ)$ by the same computation. As far as we know, these are the first non-trivial higher string topology operations
constructed. (Note also that our higher degree cycles live in components of arbitrarily high genus and number of incoming
boundary components, which should be contrasted to a result of Tamanoi \cite{tamanoi} that states that for large genus or number of incoming boundaries, the {\em degree 0} string topology 
operation on $H^*(LM)$ constructed in \cite{CohGod,godin07} are zero.)




\medskip

The paper is organised as follows: 
A technical reformulation of Theorem~\ref{natintro}
is the statement that formal operations on the Hochschild complex of $\e$--algebras are given by taking an
iterated Hochschild and then coHochschild construction on the category $\e$.
Section~\ref{sec1} defines the Hochschild and coHochschild complex, and their reduced versions, 
in the generality needed in this paper, and establishes basic properties of these, which are used in
Section~\ref{natsec} to give a proof of Theorem~\ref{natintro}. In Section~\ref{natsec} we also compare the formal operations to the natural
operations, and give the cap-product example mentioned earlier.  
In Section~\ref{compsec}  we identify the formal operations in the case of open field
theories, Frobenius algebras, and $\Ai$--algebras, proving in particular Theorems~\ref{OCintro} and \ref{SDintro},  building on our joint work with
Westerland \cite{WahWes08}. The main ingredient in the proof in each case is a computation of the homology of the relevant iterated coHochschild constructions in terms of chain complexes associated to certain 
cosimplicial sets of partitions. In  the process we need the general fact that the homology of the canonical chain complex associated to a cosimplicial set is
concentrated in degree 0, the proof of which we postpone to Section~\ref{cosimpsec}.  
Finally Section~\ref{non-trivial} gives examples of non-trivial operations on the Hochschild complex of Frobenius algebras, and in particular of
non-trivial higher string operations. 
Section~\ref{funsec} sets up the notations and conventions about functors and dg-categories used in the paper, and the final Section~\ref{GraphApp}
recalls how to define complexes of fat graphs, and exemplifies their use in the context of $\Ai$--algebras. 

\medskip

\noindent
{\em Acknowledgment.} This paper is the continuation of the author's joint work with Craig Westerland \cite{WahWes08}. In fact, this paper was started
as an elaboration on an suggestion of Craig that the coHochschild complex could play a role. Many other people have influenced this paper 
with comments, questions and answers.  In
particular, I would like to thank Alexander Berglund, Bill
Dwyer, Beno\^it Fresse, Anssi Lahtinen and Bruno Vallette. I would furthermore like to thank the referee for many helpful suggestions. 
The author
was supported  by the Danish National Sciences Research Council  (DNSRC)  and
the European Research Council (ERC), as well as by
the Danish National Research Foundation through the Centre for Symmetry and Deformation (DNRF92). 
The author would also like to thank IHES for its hospitality during the final stages of the redaction of this article.

\tableofcontents

\setcounter{section}{-1}

\section{Functors and dg-categories; notations and conventions}\label{funsec}

We work throughout this paper with dg-categories, that is categories enriched over chain complexes, by which we mean differential graded $\Z$--modules. These are categories $\C$ whose morphism sets $\C(n,m)$ are chain complexes and such that composition in the category is given by chain maps $$\C(n,p)\ot \C(p,m)\rar \C(n,m).$$  
We will denote by $\Comp$ the dg-category of chain complexes, i.e. the category whose objects  are chain complexes and whose morphisms are linear maps
of any degree. We denote by $\Hom(V,W)$ the chain complexe of morphisms from $V$ to $W$ in $\Comp$. The differential on $\Hom(V,W)$ is defined so that the evaluation map $V\ot \Hom(V,W)\to W$ is a chain map for every pair of
chain complexes $V,W$. Explicitly, for $f\in \Hom(V,W)$, this gives $df(v)=(-1)^{|v|}(d_W(f(v))-f(d_V(v)))$. In particular, chain maps $V\to W$ are the
degree 0 cycles in $\Hom(V,W)$. 

A dg-functor $F:\C\to \Comp$ is an enriched functor, that is a functor such that the maps 
$$F(n)\ot \C(n,m)\rar F(m)$$
are chain maps. 

Note that our convention in this paper is that morphisms act on the right. 
This has an influence on the signs we work with. If a reader wants to compare our signs to those one would obtain having the morphisms act on the
left, this can be done by  
 multiplying with the Koszul signs coming from permuting the factors appropriately before and, if relevant, after the map or identification considered.

\medskip

Let $\C$ be a small dg-category.  
Given dg-functors $F,G:\C\to\Comp$, let $$\Hom_\C(F,G)\ \ \subset\ \ \prod_{k\in Obj(\C)}\Hom(F(k),G(k))$$ 
denote the chain complex of natural transformations $F\to G$. 

Given dg-functors $F:\C\to\Comp$ and $G:\C^{op}\to\Comp$, we denote by 
$$F\ot_{\C}G=\bigoplus_{k\in Obj(\C)} F(k)\ot G(k)/\sim$$
 the tensor product of $F$ and $G$, where the equivalence relation is given by $F(f)(x)\ot y\sim (-1)^{|y||f|}x\ot G(f)(y)$ for any
$x\in F(k)$, $y\in G(l)$ and $f\in \C(k,l)$. This is a chain complex with differential $d=d_F+d_G$ (with the usual Koszul sign convention).

\medskip

In this paper, functors describing algebras will be monoidal functors. Note that those will always be {\em strong monoidal} (also called split monoidal), i.e. satisfying that $F(n)\ot F(m)\to F(n+m)$ is an isomorphism. 

\medskip 

By a {\em quasi-isomorphism of functors} with values in chain complexes, we mean a natural transformation given by pointwise quasi-isomorphisms.

\section{The Hochschild and coHochschild complexes}\label{sec1}

In this section, we recall from \cite{WahWes08} our generalization of the Hochschild complex of a dg-algebra, and define its dual, a generalization of the coHochschild
complex of a dg-coalgebra. We then show that our coHochschild complex is homotopy invariant, and show the equivalence between the reduced and unreduced
(co)Hochschild constructions when the algebras have units. The homotopy invariance of the Hochschild complex is proved in
\cite[Prop.~5.7]{WahWes08}.

\bigskip

We recall briefly the prop $\Ai$, already mentioned in the introduction. We will use the language of graphs recalled in the Appendix Section~\ref{GraphApp}. 
$\Ai$ is a dg-category with objects the natural numbers $\N=\{0,1,2,\dots,\}$ and with morphisms $\Ai(n,m)$ the free
$\Z$--module on the  set of graphs which are 
unions of $m$ planar trees, with a total of $n=n_1+\dots+n_m$ {\em incoming} labeled leaves, with each $n_i\ge 1$, in addition to the root of each
tree,  considered here as an outgoing leave.  (See Figure~\ref{Aimk}(a) for an example.) From $n=0$, there is only the identity morphism. As detailed in Section~\ref{GraphApp}, $\Ai(n,m)$ can be given the structure of a chain complex using the valence minus 3 of vertices to define  the degree, and the sum of all possible blow-ups to define the differential.  
Composition in $\Ai$ is defined by gluing outgoing leaves to incoming leaves with the same label, and
disjoint union of trees defines a symmetric monoidal structure on $\Ai$. 

Recall that an $\Ai$--algebra is a dg-module $A$ equipped with multiplications $m_2,m_3,\dots$ where $m_k$ is a degree $k-2$ map $m_k:A^{\ot k}\to A$. These maps satisfy relations among them which make $m_2$ into a multiplication associative up to all higher homotopies.   
The category defined above is called $\Ai$ because symmetric monoidal functors $\Phi:\Ai\to\Comp$ correspond exactly to 
(non-unital) $\Ai$-algebras. Such a functor $\Phi$ will satisfy $\Phi(n)=A^{\ot n}$ for some dg-module $A$, and the multiplication $m_k$ is obtained by evaluating $\Phi$ on the morphism in $\Ai(k,1)$ defined by a tree with a single vertex (see Figure~\ref{Aimk}(b)). The relations satisfied by the $m_k$'s follow from the dg-structure of $\Ai$ and the fact that $\Phi$ is a dg-functor.  (See Example~\ref{m3ex} and e.g.~\cite[3.1]{WahWes08} and \cite[C.2 and 9.2.7]{LodVal} for more details.)
\begin{figure}[h]
\begin{lpic}{Aiandmk(0.62,0.62)}
\lbl[b]{55,0;(a)}
\lbl[b]{155,0;(b)}
\end{lpic}
\caption{Morphism in $\Ai(6,2)$ and the map $m_k\in\Ai(k,1)$.}\label{Aimk}
\end{figure}

Dually, an $\Ai$--coalgebra is a functor $\Psi:\Ai^{op}\to \Comp$. 

\bigskip

We will define the (co)Hochschild complex of a (co)algebra using a functor 
$$\LL:\Ai^{op}\rar \Comp$$
which we now define. As a graded module, $\LL(m)=\bigoplus_{n\ge 1}\Ai(m,n)[n-1]$, where $[n-1]$ indicates a shift in degree by $[n-1]$, and $\Ai$ acts on $\LL$ by precomposition. (Here we use the notation $V[n]$ for the graded module defined by $V[n]_*=V_{*-n}$.) 
To describe the differential, we write $\LL(m)$ as   
$$\LL(m)=\bigoplus_{n\ge 1}\Ai(m,n)\ot L_n$$
where $L_n=\lgl l_n\rgl$ the free module on a single generator $l_n$ in degree $n-1$. 
The differential on $\LL(m)$ is the sum of the differential $d_{\Ai}$ of the first factor with a twisted differential $d_L$ coming from the second factor, which we describe now, first pictorially and then algebraically.
  
The pictorial way to define the differential is to think of the generator $l_n$ of $L_n$ as an oriented fat graph with a single (white) vertex and $n$ ordered leaves attached to it. The differential of $l_n$ is then the sum of all ways of blowing up that vertex, just like we blew up trees in $\Ai$, with the only difference that we blow up a white vertex as a pair of a white and a black vertex, and that the white vertex is allowed to have valence 1 or 2. Figure~\ref{dl3} shows the graph $l_3$ and all its blow-ups. 
\begin{figure}[h]
\includegraphics[width=\textwidth]{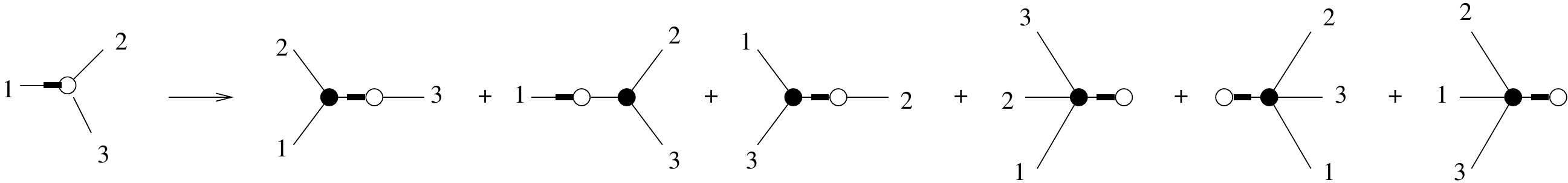}
\caption{Differential applied to $l_3$.}\label{dl3}
\end{figure}
As a blow-up of $l_n$ will always be of the form a tree attached to an $l_k$ with $1\le k<n$, we can write the differential as a map 
$$d_L:L_n\to \bigoplus_{1\le k<n} \Ai(n,k)\ot L_k.$$
This then induces a self-map of $\LL(m)$ using composition in $\Ai$: 
 $$d_L:\LL(m)=\bigoplus_{n\ge 1}\Ai(m,n)\ot L_n\to \bigoplus_{n,k\ge 1}\Ai(m,n)\ot\Ai(n,k)\ot L_k\to \bigoplus_{k\ge 1}\Ai(m,k)\ot L_k$$

We can describe $d_L$ explicitly in terms the $\Ai$--multiplications $m_k$. 
We give here the description without signs. The signs are more easily thought of as orientations of graphs. (See Example~\ref{m3ex} for how to interpret graph orientations as signs.)
We can decompose the differential as 
$$d_L(l_n)=\sum_{k<n}f_{n,k}\ot l_k$$  where $f_{n,k}\in \Ai(n,k)$ 
is itself given as 
\begin{equation}\label{fnk}
f_{n,k}=\sum_{i=1}^n\, m^i_{n,k} = \sum_{i=1}^n\pm\, m^i_{n-k+1} 
\end{equation}
with $m^i_{n-k+1}\in \Ai(n,k)$ the morphism that multiplies with $m_{n-k+1}$ the entries $i,i+1,\dots,i+n-k$ (considered modulo $n$) and $m^i_{n,k}=\pm m^i_{n-k-1}$ defined to contain the sign/orientation it receives in the differential.  The term $m^i_{n,k}$ corresponds, in the pictorial description, to the graph obtained from  $l_n$ by blowing up its leaved labeled  $i,i+1,\dots,i+n-k$ (modulo $n$) to form a tree with one vertex attached to $l_k$ at the $i$th position. 


\begin{rem}\label{LLgraphs}{\rm 
We can now think of an element of $\LL(m)$ as a sum of graphs with $m$ leaves obtained by attaching trees (the elements of $\Ai(m,n)$) to a white vertex ($l_n$), as in Figure~\ref{Ann}. Such graphs are particular examples of  black and white graphs in the language of \cite{WahWes08} and the differential of $\LL(m)$ is the differential of black and white graphs as defined in Section 2 of that paper. (See also Section~\ref{OCsec}.)}
\begin{figure}[h]
\begin{lpic}{Ann(0.62,0.62)}
\end{lpic}
\caption{Element in $\LL(12)$, decomposable as
an element of $\Ai(12,5)\ot L_5$.}\label{Ann}
\end{figure}
\end{rem}


We are now ready to define our version of the  Hochschild complex. 
Let $\e$ be a monoidal dg-category. Given a dg-functor $\Phi:\e\to\Comp$ and an object $m\in \e$, we can define a new functor 
$$\Phi(-+m):\e\to\Comp$$ by setting 
$\Phi(-+m)(n)=\Phi(n+m)$ and $\Phi(-+m)(f)=\Phi(f+id_m)$. Note that for any morphism $g\in \e(m,m')$, $\Phi(id+g)$ induces a natural
transformation $\Phi(-+m)\to \Phi(-+m')$. 

\smallskip

Recall that we call a pair $\e=(\e,i)$ a prop with $\Ai$--multiplication if $\e$ is a symmetric monoidal dg-category and $i:\Ai\to\e$ is a
symmetric monoidal dg-functor which is the identity
on objects. 

\begin{Def}\label{HcHDef} Let $(\e,i)$ be a prop with $\Ai$--multiplication. 
For a dg-functor $\Phi:\e\to \Comp$, define its {\em Hochschild complex} as the dg-functor $C(\Phi):\e\to\Comp$ given on objects by 
$$C(\Phi)(m):=i^*\Phi(-+m)\ot_{\Ai}\LL$$
and on morphisms by  
$$C(\Phi)(f):=i^*\Phi(id+f)\ot_{\Ai} id.$$
Dually, for a dg-functor $\Psi:\e^{op}\to \Comp$, we 
define its {\em coHochschild complex} as the dg-functor
$D(\Psi):\e^{op}\to\Comp$ given on objects by  
$$D(\Psi)(m):=\operatorname{Hom}_{\Ai^{op}}\big(\LL,(i^{op})^*\Psi(-+m)\big)$$
and on morphisms by  
$$D(\Psi)(f)(\underline{h}):=(i^{op})^*\Psi(id+f)\circ\underline{h}.$$
\end{Def}

The Hochschild and coHochschild complexes are natural in $\Phi$ and $\Psi$ and hence define self-maps of the functor categories:
$$C\colon\Fun(\e,\Comp)\to \Fun(\e,\Comp)\ \ \ \ \textrm{and}\ \ \ \ D\colon \Fun(\e^{op},\Comp)\to\Fun(\e^{op},\Comp).$$ 
We denote by $C^p=C\circ\dots\circ C$ and $D^p=D\circ\dots\circ D$ the iterated functors.  

\smallskip

We can describe $C(\Phi)$ and $D(\Psi)$ more explicitly as follows.  Since 
$\LL$ is quasi-free, 
$$\begin{array}{rcccl}C(\Phi)(m) &\cong& \bigoplus_{n\ge 1} \Phi(n+m)\ot L_n &\cong&  \bigoplus_{n\ge 1} \Phi(n+m)[n-1]\\
D(\Psi)(m)&\cong& \prod_{n\ge 1}\Hom\big(L_n,\Psi(n+m)\big)&\cong& \prod_{n\ge 1}\Psi(n+m)[1-n]
\end{array}$$
 as a graded modules, 
where the second isomorphism in each case comes from the fact that each $L_n$ is generated by a single element in degree $n-1$. 
The differential in the first case is given, for $x\in\Phi(n+m)$, by 
$$d(x\ot l_n)=d_\Phi x\ot l_n + (-1)^{|x|}\sum_{k=1}^{n-1} \Phi\big(i(f_{n,k})+id_m)\big)(x)\ot l_k$$
with $f_{n,k}$, $k<n$,  the terms of the differential of $L_n$ as in equation~(\ref{fnk}) above. 
In the second case, an element $\underline{h}$ of 
degree $d$ in $D(\Psi)(m)$ is determined by a sequence  
$(\u{h}(l_1),\u{h}(l_2),\dots) \in \prod_{n\ge 1}\Psi(n+m)_{n-1+d}$. In this notation,  
the differential is given by 
 $$d\u{h}(l_n)=(-1)^{n-1}\Big(d_\Psi(\u{h}(l_n))-\sum_{k=1}^{n-1}\Psi\big(i(f_{n,k})+id_m)\big)(\u{h}(l_{k}))\Big).$$

\begin{rem}{\rm 
For $\Phi:\Ai\to\Comp$ strong symmetric monoidal, with $\Phi(n)=A^{\ot n}$ for $A$ a dg-($\Ai$)-algebra, the value of $C(\Phi)$ at 0 is the Hochschild complex
of $A$: 
$$C(\Phi)(0)=\bigoplus_{n\ge 1}\Phi(n)=\bigoplus_{n\ge 1}A^{\ot n}$$
with differential the sum of the differential of $A$ and the Hochschild differential which multiplies, with higher and higher multiplications, all possible pairs, triples, quadruples,..., of consecutive factors thought of as cyclically arranged on a circle (see \cite[Rem 5.2]{WahWes08}). 

Similarly, when  $\Psi:\Ai^{op}\to \Comp$ is the functor associated to a coalgebra $C$ concentrated in degree 0, i.e. $\Psi(n)=C^{\ot n}$ with
$\Ai^{op}$-structure given by the comultiplication of $C$, the chain complex  
$D(\Psi)(0)$ recovers the classical coHochschild complex of $C$.
Indeed, we have 
$$D_*(\Psi)(0)\cong \prod_{n\ge 1}C^{\ot n}[-n+1].$$
If $C$ is concentrated in degree 0, then $D(\Psi)(0)$ 
is only non-zero in non-positive degrees, with elements of degree $-d$ of the form $(0,\dots,0,a_0\ot\dots\ot
a_d,0,\dots)$. The
differential $d=d_L$ takes such an element to the element of degree $-d-1$ whose only non-zero term is 
$$(-1)^{d}\sum_{i=0}^{d}(-1)^ia_0\ot\dots \ot(\textstyle{\sum} a_i'\ot a_i'')\ot \dots\ot a_{d} - \textstyle{\sum}a_0''\ot a_1\ot \dots\ot a_d\ot a_0'$$
where $\sum a_i'\ot a_i''$ is the comultiplication of $a_i$ (see for example \cite[Sec 3.1]{Doi81}).  
}\end{rem}

\begin{ex}{\rm \label{coex}
Suppose $A$ is an $\Ai$--algebra, and let $\Phi:\Ai\to \Comp$ with $\Phi(n)=A^{\ot n}$ be the functor defining this $\Ai$--structure. 
Consider the functor $\Hom_A:\Ai^{op}\to\Comp$ defined by   $$\Hom_A(n)=\Hom(A^{\ot n},A)$$
and, for $f\in\Ai(m,n)$, by $\Hom_A(f):\Hom(A^{\ot n},A)\to \Hom(A^{\ot m},A)$ induced by precomposition with $\Phi(f)$.
Then the coHochschild complex of $\Hom_A$ evaluated at 0 is
$$D(\Hom_A)(0)=\prod_{n\ge 1}\Hom(A^{\ot n},A)[1-n]$$
though with a differential which is {\em not} that of the Hochschild cochains---the coHochschild differential in $D(\Hom_A)$ cyclically precomposes with
multiplications. In Section~\ref{capsec}, we will though give an embedding of the Hochschild cochain complex $C^*(A,A)$ inside the complex
$D(C(\Hom_A)(0))(0)$.  
}\end{ex}

For symmetric monoidal functors, the iterated Hochschild complex
computes tensor powers of the Hochschild complex of the associated algebra (see \cite[Prop.~5.10]{WahWes08}). For the coHochschild
complex, as a graded vector space, 
 $$D^n(\Phi)(m)=\prod_{k_1,\dots,k_n\ge 1}\Hom(L_{k_1}\ot\cdots\ot L_{k_n},\Phi(k_1+\dots+k_n+m))[n-k_1-\dots -k_n].$$
The functors $C^n(-)(m)$ and $D^n(-)(m)$ can be constructed as  homology theories associated to a union of $n$ circles and $m$ points in the same way
classical Hochschild homology is associated to the circle (in Pirashvili's much more general language of higher Hochschild homology \cite{Pir00}).

\subsection{Homotopy invariance}\label{htpysec}

Proposition 5.6 in \cite{WahWes08} shows that the Hochschild complex functor is homotopy invariant. 
We show in this section the less straightforward fact that the coHochschild complex also has this property. 

\smallskip

Let $\Psi:\e^{op}\to\Comp$ be a dg-functor. Each chain complex $D(\Psi)(m)$ admits a natural filtration: using the identification  
$D(\Psi)(m)\cong \prod_{q\ge 1}\Psi(q+m)[1-q]$, we define a filtration
by $$F^s=  \prod_{q\ge s}\Psi(q+m)[1-q].$$
Each $F^s$ is indeed a subcomplex as 
the differential of $D(\Psi)$ is of the form $d=d_\Psi+d^L$, with $d_\Psi$ the differential of $\Psi$ not affecting $q$, and $d^L$ increasing $q$. We have 
$$D:=D(\Psi)(m)=F^1\supset F^2\supset \cdots\supset F^s\supset F^{s+1}\supset \cdots $$
is an exhaustive filtration ($D=\bigcup_sF^s$), which is moreover complete ($\lim_s D/F^s=\prod_s \Psi(s+m)=D$).
We have $F^s/F^{s+1}\cong \Psi(s+m)[1-s]$

Note that the coHochschild part of the differential $d^L=\sum_{n>q} f_{n,q}$ takes 
$F^{s}$ to $F^{s+1}$,  so that 
the spectral sequence associated to this filtration has the form 
$$E^1_{-p,q}=H_q(\Psi(p+1+m),d_\Psi) \ \ \ \Longrightarrow\ \ \ H_{q-p}(D(\Psi)(m)).$$




\begin{prop}\label{htpy} Let $(\e,i)$ be a prop with $\Ai$--multiplication and let $\Psi,\Psi'\in\Fun(\e^{op},\Comp)$.  
A quasi-isomorphism $\eta:\Psi\arsim\Psi'$ induces a quasi-isomorphism $D(\eta):D(\Psi)\arsim D(\Psi')$.  
\end{prop}

Recall that by a quasi-isomorphism of functors with values in chain complexes, we mean a natural transformation given by pointwise quasi-isomorphisms. 

\begin{proof}
A natural transformation $\Psi\to\Psi'$ induces a natural transformation $D(\Psi)\to D(\Psi')$ and we are left to show that this
natural transformation is by quasi-isomorphisms. For each $m$, we use the filtration of $D(\Psi)(m)$ and $D(\Psi')(m)$ defined above,
and the associated spectral sequence. 
A quasi-isomorphism of functors $\Psi\arsim\Psi'$ induces
an isomorphism of the $E^1$-terms of the spectral sequences. 
The result then follows from  the Eilenberg-Moore Comparison Theorem \cite[Thm 5.5.11]{WeiHomAlg}. 
\end{proof}

\subsection{Reduced complexes}\label{redsec}

Let $\Ai^+$ be the dg-category obtained from $\Ai$ by adding one generating morphism $u\in \Ai^+(0,1)$ with the relations that it acts as a unit for $m_2$, i.e.~we have $m_2\circ (u+id)=id=m_2\circ(id + u)\in \Ai^+(1,1)$, and such that $m_k(id_i+u+id_{k-i-1})=0$ for all $k>2$ and $0\le i\le k-1$. 
Then  symmetric monoidal dg-functors $\Phi:\Ai^+\to\Comp$ corresponds
exactly to $\Ai$--algebras $A=\Phi(1)$ equipped with a strict unit for the multiplication $m_2:A\ot A\to A$.
   
If the functor $i:\Ai\to\e$ extends to a functor $i^+:\Ai^+\to \e$ from the prop of unital $\Ai$--algebras, 
we define the 
{\em reduced Hochschild complex} of a functor $\Phi:\e\to\Comp$  to be the quotient of $C(\Phi)$ given on objects by  
$$\bC(\Phi)(m)=\sum_{n\ge 1}\Phi(n+m)/_{U_n}\ot L_n$$
for $U_n=\sum_{2\le i\le n}Im(u_i)$, with $u_i:\Phi(n-1+m)\to\Phi(n+m)$ introducing a unit at the $i$th position. Lemma 5.4 of
\cite{WahWes08} shows that the differential of $C(\Phi)$ induces a well-defined differential on $\bC(\Phi)$.

Similarly, for $\Psi:\e^{op}\to\Comp$ and $m\in \N$, 
the {\em reduced coHochschild complex} $\bD(\Psi)$ is the subcomplex of $D(\Psi)$ defined by  
$$\bD(\Psi)(m)=\prod_{n\ge 1}\Hom(L_n,K(n)\subset\Psi(n+m))$$
where $K(n)=\bigcap_{i\ge 2}^n\ker(u_i^{op})$.  
(The kernels $K(n)$ do not define a functor from $\Ai$, so that we cannot define the reduced complex directly as a
complex of natural transformations.) 
Proposition~\ref{diffprop} below checks that the differential of $D(\Psi)(m)$ restrict to a differential on $\bD(\Psi)(m)$.  

Note that $\bD(\Psi)$ is a subfunctor of $D(\Psi):\e^{op}\to \Comp$. This follows from the fact that, for $f\in \e^{op}(m,m')$ and any $i\le n$, we have a commutative diagram 
$$\xymatrix{\Psi(n+m)\ar[r]^-{\Psi(u_i^{op})}\ar[d]_{\Psi(id_n+f)}& \Psi(n-1+m)\ar[d]^{\Psi(id_{n-1}+f)}\\
\Psi(n+m')\ar[r]^-{\Psi(u_i^{op})} & \Psi(n-1+m')  
}$$ 
as $u_i^{op}$ could just as well be written as $u_i^{op}+id_m$ (resp.~$u_i^{op}+id_{m'}$). 
In fact, $\bD$ defines again a self-map of $\Fun(\e,\Comp)$. 

\medskip

To prove that $\bD(\Psi)$ is quasi-isomorphic to $D(\Psi)$, we will consider partially reduced Hochschild complexes as well: let 
$$\bD_{\le r}(\Psi)(m)=\prod_{n\ge 1}\Hom(L_n,K(n)_{\le r}\subset\Psi(n+m))$$
where $K(n)_{\le r}=\bigcap_{i=2}^{\max(r,n)}\ker(u_i)\subset\Psi(n+m)$. 
In particular,  $\bD_{\le 1}(\Psi)(m)=D(\Psi)(m)$ and $\bD(\Psi)(m)=\bigcap_{r\ge 1}\bD_{\le r}(\Psi)(m)$. 

\begin{prop}\label{diffprop}
The differential of $D(\Psi)(m)$ restricts to a differential on $\bD(\Psi)(m)$ and on each $\bD_{\le r}(\Psi)(m)$. 
\end{prop}

\begin{proof}
We prove the proposition for $\bD(\Psi)(m)$. The same proof applies to each $\bD_{\le r}(\Psi)(m)$. 
We write $d_\Psi+d^L$ for the differential of the coHochschild complex, with $d^L$ the coHochschild differential.
First note that $d_\Psi$ takes $\bD(\Phi)(m)$ to itself as $u_i$ is of degree 0 so
$d_{\Psi}(\Psi(u_i^{op})(x))=\Psi(u_i^{op})(d_{\Psi}x)$. So we are left to consider $d^L$.  

Recall that $d^L=\sum_{n>k}f_{n,k}^{op}$, where each $f_{n.k}=\sum m^i_{n,k}\in \Ai(n,k)$. 
For $x\in K(k)\subset\Psi(k+m)$, 
we will show that  $\Psi(f^{op}_{n,k})(x)=\Psi(f^{op}_{n,k}+id_m)(x)\in K(n)\subset
\Psi(n+m)$, i.e.~that each of the sum of terms defined by $f_{n,k}$ take $K(n)$ to $K(n)$. 
So assume  $\Psi(u^{op}_i)(x)=0$ for each $2\le i\le k$. We need to check that $\Psi(u^{op}_i)(\Psi(f^{op}_{n,k})(x))=0$ 
for each $2\le i\le n$. 
For this, we need to consider the composition $f_{n,k}\circ u_i\in\Ai^+(n-1,k)$.  
Recall from \cite[Lem 5.4]{WahWes08} (which deals with the dual situation) that $f_{n,k}\circ u_i$ has the form $\sum_j \,\pm (u_{i'}+ m_{n-k+1}^j)$ with $2\le i'\le i$ and $i'\le k$. 
Hence $(f_{n,k}\circ u_i)^{op}=\sum (f')^{op}\circ u_{i'}^{op}$ with $2\le i'\le i$ maps $x$ to 0 if $x\in K(k)$.
\end{proof}

\begin{prop}\label{conormalization}
The inclusion $\bD(\Psi)(m) \inc D(\Psi)(m)$ is a quasi-isomorphism. 
\end{prop}

\begin{proof} (This is an adaptation of the proof of \cite[Lem 16]{Jar11}. 
I would like to apologize in advance about the signs in this proof: 
they come from the behavior of orientations in the graph complex and they are not very transparent.)
The partially reduced complexes $F^r:=\bD_{\le r}(\Psi)(m)$ define a filtration of $D(\Psi)(m)$:   
$$\bD(\Psi)=\bigcap_{r\ge 1}F^r \subset \dots\subset F^r\subset F^{r-1}\subset\cdots \subset F^1=D(\Psi).$$
Using the identification $D(\Psi)(m)\cong\prod_{n\ge 1}\Psi(n+m)[1-n]$, we have   
$$F^r\cong\prod_{n\ge 1}K(n)_{\le r}[1-n].$$ 
Note that for each $r$, the first $r-1$ factors of $F^r$ and $F^{r-1}$ are identical.

To prove the proposition, we will show that the quotients $F^{r-1}/F^r$ have trivial homology. To identify these quotients,   
we use the short exact sequence 
$$F^{r}=\prod_{n\ge 1}K(n)_{\le r}[1-n]\ \rar\ F^{r-1}=\prod_{n\ge 1}K(n)_{\le r-1}[1-n]\ \sta{u}{\rar}\ B^r=\prod_{n\ge r} K(n-1)_{\le r-1}[1-n]$$
with $u$ the 0-map on the factors with $n<r$, and $u=(i^{op})^*\Psi(u_r^{op})$ on the other factors---to ease the notation, we will from now on in the proof simply write $u_i^{op}$ for $(i^{op})^*\Psi(u_r^{op})$ and more generally drop 
$(i^{op})^*\Psi$ from the notation on morphisms applied to elements of $K(n)$. 

Note that $u$ is surjective (the multiplication 
$m_{r-1,r}\in\Ai(n,n-1)$
giving a splitting) and that $F^r$ is indeed the kernel of $u$.  Hence $B^r\cong F^{r-1}/F^r$ as a graded vector space. 

We compute the quotient differential on $B^r$. 
 Recall that for $x=(x_1,x_2,\dots)\in F^{r-1}$, its differential has component $n$ given by 
$(dx)_n=(-1)^{n-1}(d_{\Psi}(x_n)-\sum_{k<n}f_{n,k}^{op}(x_k))$, where $f_{n,k}=\sum_{i=1}^{n}m^i_{n,k}$ where $m_{n,k}^i=\pm m^i_{n-k+1}$ multiplies the entries $i,\dots,i+n-k$. 
For $y=(y_r,y_{r+1},\dots)\in B^r$, with each $y_n\in K(n-1)_{\le r-1}$, 
define $dy$ by 
$$(dy)_n=(-1)^{n-1}(d_\Psi(y_n)+ \sum_{k=r}^{n-1}\sum_{i=r}^k (-1)^{n-k}(m^i_{n-1,k-1})^{op}(y_{k}))$$
To check that this is the quotient differential, we need to check that the diagram 
$$\xymatrix{F^{r-1}\ar[r]^u\ar[d]^d & B^r\ar[d]^d\\
F^{r-1}\ar[r]^u & B^r}$$
commutes. Let  $x=(x_1,x_2,\dots)\in F^{r-1}$.
Then for $n\ge r$, 
\begin{align*}(u\circ d(x))_{n}&=(-1)^{n-1}(u^{op}_r(d_{\Psi}(x_n))-\sum_{k<n}u^{op}_r(f^{op}_{n,k}(x_k)))\\
&=(-1)^{n-1}(u_r^{op}(d_{\Psi}(x_n))-\sum_{k<n}(f_{n,k}\circ u_r)^{op}(x_k))).
\end{align*}
We have $f_{n,k}\circ u_r=\sum_{i=1}^{n} m^i_{n,k}\circ u_r$. There are 3 possibilities: $r$ can be before the multiplication, at the multiplication,
or after it. If $n-k>1$, and $r$ is at the multiplication, i.e. $i\le r\le i+n-k$ (understood cyclically), then the composition is 0. If $n-k=1$, there are exactly
two multiplications which include $r$, namely $i=r-1$ and $i=r$ (again understood cyclically, i.e. mod $n$), and the two terms will cancel each other. Hence there
will be no contribution of that form. If $r$ is after the multiplication, $m^i_{n,k}\circ u_r=u_j\circ m^i_{n-1,k-1}$ for some $j< r$, and hence it acts as 0 on
$F^{r-1}$. More precisely, this happens when $r>i+n-k$ cyclically, except when $i+n-k=1$.  
Hence the only the terms contributing to this sum are those with $i>r$ and $i+n-k<n+2$ (now non-cyclically), using the fact that $r>1$ for the case $i+n-k=n+1$. 
This means that only the terms $m^i_{n,k}$ with $r<i<k+2$ contribute, which in particular requires $k\ge r$. 
Now in these cases we have $m_{n,k}^i\circ u_r=(-1)^{n-k}u_r\circ m^{i-1}_{n-1,k-1}$. 
So  $$(u\circ d(x))_{n}=(-1)^{n-1}(d_\Psi(u_r^{op}(x_n))+\sum_{k=r}^{n-1}\sum_{i-1=r}^{k}(-1)^{n-k}(m^{i-1}_{n-1,k-1})^{op}(u_r^{op}(x_k)$$
which is exactly the $n$th component of $d\circ u(x)$. 

We now show that $B^r$ has trivial homology. 
Define $s:B^r\to B^r$ of degree $+1$ by $s(y)_n=(-1)^{n+r}u_{r}^{op}(y_{n+1})$. Explicitly, we have
\begin{align*}
(sd(y))_n= & \ (-1)^{n+r}u_r^{op}(dy)_{n+1} \\
= & \ (-1)^{r}(u_{r}^{op}(d_\Psi(y_{n+1}))+\sum_{k=r}^{n}\sum_{i=r}^k(-1)^{n-k+1}u_{r}^{op}((m^{i}_{n,k-1})^{op}(y_{k})))\\
(ds(y))_n=& \ (-1)^{n-1}((-1)^{n+r}d_\Psi(u_{r}^{op}(y_{n+1}))\\
& +\sum_{k'=r}^{n-1}\sum_{i'=r}^{k'}(-1)^{n-k'+r+k'} (m^{i'}_{n-1,k'-1})^{op}(u_{r}^{op}(y_{k'+1})))\\
\end{align*}
First note that for $k=n$ and $i=r$ in the first line, we have $u_r^{op}((m^r_{n,n-1})^{op}(y_n))= (m^r_{n,n-1}\circ u_r)^{op}y_n =(-1)^{r+1}
y_n$. We would like to
show that no other term contribute to the sum $(sd(y))_n+(ds(y))_n$. 

If $i=r$ in the first sum with $k<n$,  the composition $m^r_{n,k-1}\circ u_r=0$, so such terms do not contribute. 
In particular, as $k=r$ implies $i=r$, there are no contribution of that form.  

In the first sum, we are thus left with the terms $r<k\le n$ and $r< i\le k$. Then the composition 
$m^i_{n,k-1}\circ u_r=(-1)^{n-k+1}u_r\circ m^{i-1}_{n-1,k-2}$. 
Choosing $k'=k-1$ and $i'=i-1$ gives an identification between those terms and
minus all the terms of the second sum. 

Hence $sd+ds=id$.
It follows that $B^r$ has trivial homology and thus that each inclusion $F^r\inc F^{r-1}$ is a quasi-isomorphism. 
To obtain the desired quasi-isomorphism, we consider the short exact sequence $$\bD(\Psi)=\cap_r F^r\to D(\Psi)=F^1\to \cup_r B^r.$$
As $H_*(B^r)=0$ for each $r$, we have that $H_*(\cup_r B^r)=0$ (any cycle being in some $B^r$), which implies the result. 
\end{proof}

\begin{prop}\label{normalization}
The quotient map $C(\Phi)\to\bC(\Phi)$ is a quasi-isomorphism of functors. 
\end{prop}

\begin{proof}
The proof is a dual version of the previous proof: we consider the partial quotients of $C(\Phi)(m)\cong\oplus_{n\ge 1}\Phi(n+m)$
$$F_r=\oplus_{n\ge 1}\Phi(n+m)/U_{\le r} \ \ \ \textrm{where}\ \ \  U_{\le r}=\sum_{i=2}^{max(n,r)}Im(u_i)$$
Then we have a sequence of quotient maps
$$C(\Phi)(m)=F_1\to F_2\to\dots\to F_\infty=\bC(\Phi)(m).$$
We compute the quotients $F_r/F_{r-1}$ and $F_\infty/F_1=\bC(\Phi)(m)/C(\Phi)(m)$ via short exact sequences 
$$A_r\sta{u_r}{\rar} F_{r-1}\to F_{r} \ \ \ \textrm{and}\ \ \ A_\infty=\cup_r A_r \to F_1\to F_\infty.$$
Again, we want to show that the $A_r$'s are acyclic. 
One checks that, as a graded vector space, $A_r=\oplus_{n\ge r}\Phi(n-1+m)[n-1]$ and that the quotient 
differential is defined on $x\in \Phi(n-1+m)$ by 
$$dx=d_\Phi(x)+(-1)^{|x|} \sum_{k=r}^{n-1}\sum_{i=r}^k (-1)^{n-k}m^i_{n-1,k-1}(x).$$
Then define $s:A_r\to A_r$ by $s(x)=(-1)^{|x|}u_r(x)$ if $x\in\Phi(n-1+m)$. 
This satisfies $sd+ds=id$, which shows that each $A_r$ is acyclic, and hence that their union is acyclic. 
\end{proof}

\section{Formal and natural operations on the Hochschild complex}\label{natsec}

In this section, we first give a model for the formal operations $\Nat_\e$, the natural operations on the Hochschild complex of {\em generalized $\e$--algebras}, i.e.~functors
$\e\to\Comp$ without any monoidality assumptions. In Section~\ref{monoidalsec} we then compare this chain complex to the chain complex of
natural operations $\Nat^\ot_\e$ on the Hochschild complex of actual algebras. Algebras here are algebras over an arbitrary  prop with $\Ai$--multiplication
$(\e,i)$. 
We say that $\e$ has units if $i:\Ai\to\Comp$ extends to a functor $i^+:\Ai^+\to\Comp$.

\subsection{Formal operations} 
   
For each $m,n\ge 0$, we have the dg-functor 
$$C_{\e}^{(n,m)}: \Fun(\e,\Comp)\rar \Comp$$
taking a functor $\Phi$  to the chain complex $C^n(\Phi)(m)$ of Definition~\ref{HcHDef}, the iterated Hochschild complex of $\Phi$ evaluated at $m$. 
Consider the category $\Nat_\e$ with objects pairs of natural numbers $\oc{m}{n}$ with $n,m\ge 0$, 
and morphisms the Hom-complexes of the above functors: 
$$\Nat_\e(\oc{m_1}{n_1},\oc{m_2}{n_2}):= \Hom(C_{\e}^{(n_1,m_1)},C_{\e}^{(n_2,m_2)}).$$ 
Similarly, when $\e$ has units, we define $\overline{\Nat}_\e$ as the category with morphisms the Hom-complexes of the iterated reduced Hochschild complexes
of Section~\ref{redsec}: 
$$\overline{\Nat}_\e(\oc{m_1}{n_1},\oc{m_2}{n_2}):= \Hom(\bC_{\e}^{(n_1,m_1)},\bC_{\e}^{(n_2,m_2)}).$$ 

\smallskip

The representable functors $\e(p,-):\e\to\Comp$ have a well-defined Hochschild, and $n_1$--iterated Hochschild complex $C^{n_1}(\e(p,-))$. By \cite[Prop
5.5]{WahWes08}, these functors for varying $p$'s assemble to a define functor $$C^{n_1}\e(-,-):\e^{op}\times \e\rar \Comp$$
so that it makes sense to subsequently take an $n_2$--iterated coHochschild complex construction in the first variable to obtain a new functor
$$D^{n_2}C^{n_1}\e(-,-) :\e^{op}\times \e\rar \Comp.$$

\begin{thm}\label{natural} Let $(\e,i)$ be a prop with $\Ai$--multiplication. Then 
there are isomorphisms of chain complexes 
$$\Nat_\e(\oc{m_1}{n_1},\oc{m_2}{n_2})\cong D^{n_1}C^{n_2}\e(-,-)(m_1,m_2)$$ 
and, when $\e$ has units,  
$$\overline{\Nat}_\e(\oc{m_1}{n_1},\oc{m_2}{n_2})\cong \bD^{n_1}\bC^{n_2}\e(-,-)(m_1,m_2).$$ 
Moreover in that case we have $\bNat_\e(\oc{m_1}{n_1},\oc{m_2}{n_2})\simeq \Nat_\e(\oc{m_1}{n_1},\oc{m_2}{n_2})$.  
\end{thm}

More explicitly, this says that 
$$\Nat_\e(\oc{m_1}{n_1},\oc{m_2}{n_2})\cong\!\!\prod_{j_1,\dots,\, j_{n_1}\ge 1}\ \bigoplus_{k_1,\dots,\, k_{n_2}\ge 1} 
\e(j+m_1,k+m_2)\,[k-j+n_1-n_2] $$
with $j=j_1+\dots+j_{n_1}$, $k=k_1+\dots+k_{n_2}$, and 
where the square brackets indicate the degree shift. The  
differential is the sum of the differential of $\e$ and the Hochschild and coHochschild differentials. 
Explicitly, in the case $n_1=1=n_2$ for simplicity, an element $\u{g}\in \Nat_\e(\oc{m_1}{1},\oc{m_2}{1})$ under the above
isomorphism has the form $\u{g}=\{g_j\}_{j\ge 1}$ with 
$g_j=g_{j,k_1}+\dots+g_{j,k_{r_j}}$ and $g_{j,k_i}\in\e(j+m_1,k_i+m_2)$. Its differential has $\uj$th component 
$d(\u{g})_j$ given by 
$$(-1)^{j-1}\Big(d_\e(g_j)+\sum_{i=1}^{r_j}\sum_{k=1}^{k_i-1}(-1)^{|g_{j,k_i}|}(f_{k_i,k}+id_{m_2})\circ g_{j,k_i} - \sum_{j'=1}^{j-1}g_{j'}\circ (f_{j,j'}+id_{m_1})  \Big).$$

\medskip

We say that a functor $\e\to\e'$ is a {\em quasi-isomorphism of categories} if it is the identity on objects and it induces a quasi-isomorphism on each morphism complex. 

Using the homotopy invariance of the Hochschild and coHochschild constructions (Proposition~\ref{htpy} and \cite[Prop 5.6]{WahWes08}), we immediately 
get the following corollary:

\begin{cor}\label{Nathtpy}
A quasi-isomorphism of categories $\e\arsim\e'$ induces a quasi-isomorphism $\Nat_{\e}\arsim\Nat_{\e'}$. 
\end{cor}

To prove Theorem~\ref{natural}, we first show that the above chain complex does act by natural transformations on the
Hochschild complex of generalized $\e$-algebras. 

\begin{lem}
For any $\Phi:\e\to\Comp$, there is a chain map 
$$C^{n_1}(\Phi)(m_1)\ot D^{n_1}C^{n_2}\e(-,-)(m_1,m_2)\rar C^{n_2}(\Phi)(m_2)$$
which is natural in $\Phi$. This map restricts to a map 
$$\bC^{n_1}(\Phi)(m_1)\ot \bD^{n_1}\bC^{n_2}\e(-,-)(m_1,m_2)\rar \bC^{n_2}(\Phi)(m_2)$$
in the reduced case. 
\end{lem}

\begin{proof}
To simplify notations write $\u{j}=(j_1,\dots,j_{n_1})$, $j=j_1+\dots+j_{n_1}$ and $L_{\u{j}}=L_{j_1}\ot\dots\ot
L_{j_{n_1}}$,  and similarly for $\uk=(k_1,\dots,k_{n_2})$. 
In the unreduced case, the action is defined as follows:
$$\xymatrix@R=1pc{\big(\Phi(-+m_1)\ot_{\Ai} \LL\big) \ot \Hom_{\Ai^{op}}(\LL,\e(-+m_1,-+m_2)\ot_{\Ai} \LL)
  \ar[d]\\ 
\Phi(-+m_1)\ot_{\Ai}\e(-+m_1,-+m_2)\ot_{\Ai} \LL
  \ar[d]\\ 
\Phi(-+m_2)\ot_{\Ai} \LL
}$$
where the first arrow applies the homomorphism to the $\Ai^{op}$--module $\LL$ used to define the Hochschild complex in Section~\ref{sec1}, and 
the second arrow is induced by the action of $\e$ on $\Phi$. As the action on $\Phi$ is via its $\e$-structure, 
it is natural with respect to natural transformations of functors $\Phi\to\Phi'$. 

More explicitly, the action is the composition 
$$\xymatrix@R=1pc{\big(\bigoplus_{\uj}\Phi(j+m_1)\ot L_{\uj}\big) \ot \prod_{\uj}\Hom(L_{\uj},\oplus_{\uk}\e(j+m_1,k+m_2)\ot L_{\u{k}})
  \ar[d]\\ 
\oplus_{\uj,\u{k}}\Phi(j+m_1)\ot\e(j+m_1,k+m_2)\ot L_{\u{k}}
  \ar[d]\\ 
\oplus_{\u{k}}\Phi(k+m_2)\ot L_{\u{k}}
}$$
which restricts to a well-defined action of $\bD^{n_1}\bC^{n_2}\e(-,-)(m_1,m_2)$ in the reduced case because the reduced coHochschild complex only 
involves homomorphisms with image in the kernel of precomposition by the maps
$u_i$ whose images are quotiented out in the reduced Hochschild complex of $\Phi$. 
\end{proof}

This lemma gives a map 
$$F: D^{n_1}C^{n_2}\e(-,-)(m_1,m_2)\to \Nat_\e(\oc{m_1}{n_1},\oc{m_2}{n_2})$$
(and reduced version) 
which is easily seen to be injective using as $\Phi$'s the representable functors $\e(p,-)$. 
The composition on the complexes $D^{n_1}C^{n_2}\e(-,-)(m_1,m_2)$ for varying $m_i$'s and $n_i$'s
which corresponds to the composition of natural transformations in $\Nat_\e$ under this map is of
the form 
$$\xymatrix@R=1pc{\prod_j\Hom(L_j,\oplus_k\e(j,k)\ot L_k)\ot \prod_k\Hom(L_k,\oplus_l\e(k,l)\ot L_l)\ar[d]\\
\prod_j\Hom(L_j,\oplus_{k,l}\e(j,k)\ot\e(k,l)\ot L_l)\ar[d]\\
\prod_j\Hom(L_j,\oplus_{l}\e(j,l)\ot L_l).}$$ 

\begin{proof}[Proof of Theorem~\ref{natural}]
We will construct a pointwise  inverse to the map $F$ given by the lemma above. 
For each $m_1,m_2,n_1,n_2\ge 0$, a map
\begin{align*}
G: \Nat_\e(\oc{m_1}{n_1},\oc{m_2}{n_2})\ \rar & \ \ D^{n_1}C^{n_2}\e(-,-)(m_1,m_2)\\
&=\prod_{\uj}\Hom(L_{\uj},\bigoplus_{\uk}\e(j+m_1,k+m_2)\ot L_{\uk})
\end{align*}
is determined by giving, for each natural transformation $\nu$ and each tuple $\uj=(j_1,\dots,j_{n_1})$, an element 
$G(\nu)_{\uj}(l_{\uj})\in \bigoplus_{\uk}\e(j+m_1,k+m_2)\ot L_{\uk}$, for $l_{\uj}$ the generator of $L_{\uj}$. 
To produce such an element given tuple $\uj=(j_1,\dots,j_{n_1})$, we consider the action of $\Nat_\e$ on the Hochschild complex of the 
functor $\e(j+m_1,-)$ for $j=j_1+\dots+j_{n_1}$:  \\
\mbox{\resizebox{.99\hsize}{!}{
$\xymatrix{{ C^{n_1}(\e(j+m_1,-))(m_1) \ot \Nat_\e(\oc{m_1}{n_2},\oc{m_2}{n_2})} \ar[d]_= \ar[r]^-\al & C^{n_2}(\e(j+m_1,-))(m_2) \\
\big(\bigoplus_{\uj'}\e(j+m_1,j'+m_1)\ot L_{\uj'}\big)\ot \Nat_\e(\oc{m_1}{n_2},\oc{m_2}{n_2})\ar[r]&
\bigoplus_{\uk}\e(j+m_1,k+m_2)\ot L_{\uk}. \ar[u]^=}$}}\\
We define $$G(\nu)_{\uj}(l_{\uj}):=\nu(id_{j+m_1}\ot l_{\uj})=\al((id_{j+m_1}\ot l_{\uj})\ot \nu).$$

We check that $G$ is a chain map: as $\al$ is a chain map, we have that 
$G(d\nu)_{\uj}(l_{\uj})=(-1)^{j-1}(d(\nu(id_{j+m_1}\ot l_{\uj}))-\nu(d(id_{j+m_1}\ot l_{\uj})))$. The first term
gives exactly the differential of the Hochschild complex of $\e$, so we are left to identify the second term with the coHochschild
differential. We have $d(id_{j+m_1}\ot l_{\uj})=\sum_{\u{j'}<\uj}f_{\uj,\u{j'}}\ot l_{\u{j'}}$, where $\u{j'}<\uj$ are the tuples
$(j'_1,\dots,j_{n_1}')$ with $j'_i=j_i$ expect for one $i=s$ where $j_s'<j_s$, and 
$f_{\uj,\u{j'}}=(-1)^{j_1+\dots+j_{s-1}-s+1}id+\dots+ f_{j_s,j'_s}+
\dots+ id$. 
By naturality, $\nu(\sum_{\u{l}}f_{\uj,\u{l}}\ot l_{\u{l}})=\sum_{\u{l}}(\nu(id_{l+m_1}\ot l_{\u{l}})\circ (f_{\uj,\u{l}}+id_{m_1}))$, which is the
coHochschild differential. 

One easily checks that $G\circ F$ is the identity on $D^{n_1}C^{n_2}\e(-,-)(m_1,m_2)$: an element in
$\u{g}\in \prod_{\uj}\Hom(L_{\uj},\bigoplus_{\uk}\e(j+m_1,k+m_2)\ot L_{\uk})$ defines a natural transformation via the map $F$ which takes
$id_{j+m_1,-}\ot l_{\uj}$ to its own $j$th component $g_j$. To check that $F\circ G$ is the identity
on $\Nat_\e(\oc{m_1}{n_2},\oc{m_2}{n_2})$, we need to check that the natural transformation $F\circ G(\nu)$ acts the same way as
$\nu$ on any functor $\Phi:\e\to\Comp$. 
 Consider an element $x\ot l_{\uj}\in C^{n_1}(\Phi)(m_1)=\oplus_{\uj}\Phi(j+m_1)\ot L_{\uj}$.  
There is an associated natural transformation
$\nu_x:\e(j+m_1,-)\to\Phi$ defined by evaluation at $x$ which takes $id_{j+m_1}$ to $x$. 
By naturality of the natural transformations, we have a commutative diagram 
$$\xymatrix{C^{n_1}(\e(j+m_1,-))(m_1) \ot \Nat_\e(\oc{m_1}{n_1},\oc{m_2}{n_2}) \ar[r]^-{\al} \ar[d] &  C^{n_2}(\e(j+m_1,-))(m_2) \ar[d]\\
C^{n_1}(\Phi)(m_1) \ot \Nat_\e(\oc{m_1}{n_1},\oc{m_2}{n_2}) \ar[r]^-{\al} & C^{n_2}(\Phi)(m_2) 
}$$
where the vertical arrows are induced by the natural transformations.  
By definition, the actions of $\nu$ and $F\circ G(\nu)$ are the same on the elements of the type $id_{j+m_1}\ot l_{\uj}$, 
and thus they agree on $x\ot l_{\uj}$ by 
the commutativity of the
above diagram. 

Note again that the inverse $G$ restrict to an inverse in the reduced case. 

For the quasi-isomorphism between the reduced and unreduced cases, we use 
the sequence of quasi-isomorphisms  
$$\xymatrix{\bD^{n_1}\bC^{n_2}\e(-,-)(m_1,m_2) \ar[r]^{\simeq} & D^{n_1}\bC^{n_2}\e(-,-)(m_1,m_2)\\
& D^{n_1}C^{n_2}\e(-,-)(m_1,m_2) \ar[u]_\simeq}$$
given by first applying Proposition~\ref{conormalization} iteratively, together with Proposition~\ref{htpy}, and then
Proposition~\ref{normalization} together with \cite[Prop 4.5]{WahWes08}. 
\end{proof}

\begin{rem}{\rm 
The category $\Nat_\e$ is an {\em extension of the Hochschild core category of $\e$} in the sense of \cite[Sec 5.2]{WahWes08} as it satisfies that 
$$\Nat_\e(\oc{m_1}{0},\oc{m_2}{n_2})\cong C^{n_2}(\e(m_1,-))(m_2).$$
Now by Corollary~5.12 of \cite{WahWes08}, any extension $\widetilde\e$ of $\e$ admits a functor $\widetilde\e\to\Nat_\e$ as it acts by natural
transformations on the Hochschild complex of functors $\Phi:\e\to\Comp$. Another characterization of $\Nat_\e$ is that it is the universal extension of $\e$.  
}\end{rem}

\subsection{Restriction to natural operations}\label{monoidalsec}

Recall from the introduction the category $\Fun^\ot(\e,\Comp)$ of $\e$--algebras whose objects are strong symmetric monoidal dg-functors 
$\Phi:\e\to\Comp$, i.e.~with natural isomorphisms 
$\Phi(n)\ot\Phi(m)\sta{\cong}{\rar}\Phi(n+m)$ compatible with the symmetries in $\e$, and whose morphisms are natural transformation $\theta:\Phi\to
\Psi$  satisfying
$$\theta_n=(\theta_1)^{\ot  n}\colon\ \Phi(n)\cong\Phi(1)^{\ot n}\rar \Psi(n)\cong\Psi(1)^{\ot n}$$ 
with $\theta_1:\Phi(1)\to \Psi(1)$ a chain map. 
 
\smallskip

Writing $C_\e^{\ot (n,m)}$ for the functor taking $\Phi$ to 
$\big(C(\Phi)(0)\big)^{\ot n}\ot \big(\Phi(1)\big)^{\ot m}$,  Proposition~5.10 of \cite{WahWes08} says that the diagram  
$$\xymatrix{\Fun(\e,\Comp) \ar[rr]^-{C_\e^{(n,m)}} && \Comp\\
\Fun^\ot(\e,\Comp) \ar@{^(->}[u] \ar[urr]_{\ \ \ C_\e^{\ot (n,m)}} &
&
}$$
commutes. Note that $\Fun^{\ot}(\e,\Comp)$ is not a dg-category as such, but it can be replaced, if one wants to stay in the world of dg-categories, by the sub-dg-category of $\Fun(\e,\Comp)$ it generates, that is allowing sums of monoidal natural transformations even when these are not themselves monoidal. 
Let 
$$\Nat_\e^\ot(\oc{m_1}{n_1},\oc{m_2}{n_2}):=\Hom\big(C_\e^{\ot (n_1,m_1)},C_\e^{\ot (n_2,m_2)}\big)$$
denote the chain complex of natural transformations of these functors. 
There is  a restriction map 
$$r: \Nat_\e(\oc{m_1}{n_1},\oc{m_2}{n_2}) \rar \Nat^\ot_\e(\oc{m_1}{n_1},\oc{m_2}{n_2})$$
as a natural transformation on all functors gives in particular a natural transformation on the subcategory of symmetric monoidal functors. 
We will show in this section that the restriction map $r$ is an isomorphism
provided that $\e$ is ``complete'' in a sense defined below.

\medskip

Recall that $\Nat_\e(\oc{m_1}{0},\oc{m_2}{0})\cong \e(m_1,m_2)$. On the other hand, 
 $\Nat^\ot_\e(\oc{m_1}{0},\oc{m_2}{0})=\Hom(U^{\ot m_1},U^{\ot m_2})$ for 
$U:\Fun^{\ot}(\e,\Comp)\to \Comp$ the forgetful functor from the category of $\e$--algebras to the category of chain complexes taking $\Phi$ to $\Phi(1)$, and
$U^{\ot m_i}$ its tensor powers. 
 So when 
$n_1=0=n_2$, the restriction map $r$ above is the representation map of the prop $\e$, i.e.~the map 
$$\rho: \e(m_1,m_2) \rar \Hom(U^{\ot m_1},U^{\ot m_2})$$
which associates to a morphism of $\e$ the operation it induces on all $\e$--algebras.

\begin{Def}\label{compdef}
We say that a prop $\e$ is {\em complete} if the representation functor $\rho$ defined above is an isomorphism. If $\e$ is not complete, we call the
target of $\rho$, the prop of natural transformations of the forgetful functors, its {\em completion}, and denote it $\widehat\e$. 
\end{Def}

So given a prop $\e$,  its completion $\widehat\e$ is a prop defined by 
$$\widehat\e(m_1,m_2)=\Hom(U^{\ot m_1},U^{\ot m_2})=\Nat_\e^\ot(\oc{m_1}{0},\oc{m_2}{0})$$
with $U$ as above.

\begin{lem}
Let $\e$ be a prop and $\widehat\e$ its completion. Then their categories of algebras are isomorphic:  $\Fun^\ot(\e,\Comp)\cong\Fun^\ot(\widehat\e,\Comp)$.
\end{lem}

\begin{proof}
The representation map $\rho$ above induces a functor 
$$\rho^*\colon \Fun^\ot(\widehat\e,\Comp)\rar\Fun^\ot(\e,\Comp)$$ by precomposition. 
On the other hand, any $\e$--algebra is an $\widehat\e$--algebra by definition of $\widehat\e$, hence this functor is the identity on objects. Now morphisms of $\e$/$\hat\e$--algebras are chain maps $A\to B$ that commute with the maps in $\e$/$\hat \e$. This implies directly that the functor is injective on morphisms as both $\e(A,B)$ and $\hat\e(A,B)$ are subsets of $\Hom(A,B)$ and the functor commutes with the inclusions. Surjectivity of $\rho^*$ follows from the definition of $\hat\e$: the morphisms in $\hat\e$ are precisely the natural transformations that commute with all maps of $\e$--algebras, so maps of $\e$--algebras are also maps of $\hat\e$--algebras. 
\end{proof}

As the completion of a prop is defined in terms of its category of algebras $\Fun^\ot(\e,\Comp)$, 
the lemma immediately implies the following 

\begin{cor}
The completion of a prop is complete. 
\end{cor}

Similarly, when $\e$ is a prop with $\Ai$--multiplication, we get

\begin{cor}\label{compiso1}
Let $(\e,i)$ be a prop with $\Ai$--multiplication and $(\widehat\e,\rho\circ i)$ its completion. Then 
 $$\Nat^\ot_\e(\oc{m_1}{n_1},\oc{m_2}{n_2}) \cong \Nat^\ot_{\widehat\e}(\oc{m_1}{n_1},\oc{m_2}{n_2})$$
\end{cor}

We are now ready to prove the main result of this section, which says that 
the injectivity and surjectivity of the restriction map $r$ is
determined by these conditions on the representation map $\rho$: 

\begin{thm}\label{compthm} Let $(\e,i)$ be a prop with $\Ai$--multiplication.  The restriction map 
$r: \Nat_\e \rar \Nat^\ot_\e$ is an isomorphism if and only if $\e$ is complete. 

More precisely, let $(\widehat\e,\rho\circ i)$ denote the completion of $\e$. Then $r$ is injective (resp.~surjective) if and only if the map $\rho:\e\to\widehat\e$
is injective (resp.~surjective).  
\end{thm}

In particular, we can always write the natural operations on the Hochschild complex of  $\e$--algebras in terms of the generalized natural operations
 for its completion $\widehat\e$: 

\begin{cor}\label{compcor}
Let $(\e,i)$ be a prop with $\Ai$--multiplication and $(\widehat\e,\rho\circ i)$ its completion. Then 
$$\Nat^\ot_\e(\oc{m_1}{n_1},\oc{m_2}{n_2}) \cong \Nat_{\widehat\e}(\oc{m_1}{n_1},\oc{m_2}{n_2})$$
\end{cor}

This follows from Theorem~\ref{compthm} and from Corollary~\ref{compiso1}.

\begin{proof}[Proof of Theorem~\ref{compthm}]
Let $\nu\in \Nat_\e(\oc{m_1}{n_1},\oc{m_2}{n_2})$ and suppose that $r(\nu)=0$. Then for any $\e$--algebra $\Phi\in \Fun^\ot(\e,\Comp)$, the map 
$\nu_\Phi:C^{n_1}(\Phi)(m_1)\to C^{n_2}(\Phi)(m_2)$ is the 0-map. We have 
$$\nu=(\nu_{\uj})\in \prod_{\uj}\bigoplus_{\uk}\Hom(L_{\uj},\e(j+m_1,k+m_2)\ot L_{\uk})\cong \prod_{\uj}\bigoplus_{\uk}\e(j+m_1,k+m_2)[k-j+n_1-n_2]$$ and $\nu_\Phi$ is defined 
on $\Phi(j+m_1)\ot L_{\uj}\subset C^{n_1}(\Phi)(m_1)$ as the composition
$$\xymatrix{\Phi(j+m_1)\ot L_{\uj} \ar[r]^-{\nu_{\uj}}\ar[rd]_{\nu_\Phi} & \bigoplus_{\uk} \Phi( j+m_1)\ot\e(j+m_1,k+m_2)\ot L_{\uk} \ar[d]\\
 & \bigoplus_{\uk}\Phi(k+m_2)\ot L_{\uk} 
}$$
By the above diagram, $\nu_\Phi$ acting as 0 for each $\Phi$ is equivalent to each $\nu_{\uj}(l_{\uj})_{\uk}\in  \e(j+m_1,k+m_2)$ 
acting as 0 for each $\Phi$, i.e. $\rho(\nu_{\uj})=0$ for each
$\uj$. But that implies that each $\nu_{\uj}(l_{\uj})_{\uk}=0$ when $\rho$ is injective, and hence that $\nu=0$, which proves the injectivity part of
the statement. 

\medskip

Consider now $\nu\in\Nat^\ot_\e(\oc{m_1}{n_1},\oc{m_2}{n_2})$. Then for each $\e$--algebra $\Phi$, we have a map 
$$\nu_\Phi:\bigoplus_{\uj}\Phi( j+m_1)\ot L_{\uj}\rar \bigoplus_{\uk}\Phi( k+m_2)\ot L_{\uk}. $$
This map is given by a collection of maps 
$(\nu_\Phi)_{(\uj,\uk)}:\Phi( j+m_1)\rar\Phi(k+m_2)$ which are natural in $\Phi$. In other words these maps define an element  
$(\nu)_{(\uj,\uk)}\in \widehat\e(j+m_1,k+m_2)$. If $\rho:\e\to\widehat\e$ is surjective,  this implies that $(\nu)_{(\uj,\uk)}=\rho(f)$ for some 
$f\in \e(j+m_1,k+m_2)$, and we have that $(\nu_\Phi)_{(\uj,\uk)}=\Phi(f)$ for each $\Phi$. Hence $\nu$ comes from an element in $\Nat_\e(\oc{m_1}{n_1},\oc{m_2}{n_2})$. 
\end{proof}

\medskip

The representation map $\rho:\e\to\widehat\e$ is injective exactly when 
any two morphisms in $\e$ can be distinguished by their actions on $\e$--algebras,  i.e.~if for every 
$j,k$ and $f\neq g\in\e(j,k)$, there exists a  $\Phi\in \Fun^\ot(\e,\Comp)$ such that $\Phi(f)\neq\Phi(g)$. 
This is always possible when $\e$ admits free algebras, which is the case when $\e$ is the prop associated to an operad:


\begin{ex}\label{opex1}{\rm 
Suppose $\e_{\pp}$ is the prop associated to an operad $\pp=\{\pp(n)\}_{n\ge 0}$, so
$$\e_{\pp}(p,q)=\bigoplus_{\begin{subarray}{c}
n_1+\dots+n_q=p\end{subarray}}\pp(n_1)\ot \cdots\ot \pp(n_q) \ot_{\Si_{n_1}\x\cdots\x\Si_{n_q}}\Z\Si_{p}.$$ 
Let $\widehat\e_{\pp}$ denote its completion. Then the representation map $\rho\colon \e_{\pp}\to \widehat\e_{\pp}$ is injective as any two elements
of  $\e_{\pp}$ can be distinguished by their actions on some free
algebra $\pp(X)=\bigoplus_n \pp(n)\ot_{\Si_n}X^{\ot n}$. }
\end{ex}

The representation map is however not injective in general, as shown by the following example: 

\begin{ex}{\rm  (lunch time example of Oscar Randal-Williams.) 
Let $\operatorname{Frob}$ be the prop of commutative Frobenius algebras 
(i.e.~the prop of (closed) 2-dimensional topological quantum field theories, see e.g.~\cite[1.2]{Koc04}) over a field. 
Consider the quotient $\operatorname{Frob}\!/_{(T=0)}$ of $\operatorname{Frob}$ by the relation that the operation associated to the torus equals the 0-map. 
As the torus operation in the Frobenius prop is multiplication by the dimension of the algebra (see \cite[1.2.29]{Koc04}), this means that only 
a 0-dimensional Frobenius algebra  can be an algebra over this prop.  
Hence there is only one $\operatorname{Frob}\!/_{(T=0)}$--algebra, the trivial 0-dimensional one, 
and elements in the prop cannot be distinguished by their action on algebras.  
}\end{ex}





Surjectivity of the representation map is more difficult to check. It is at least satisfied in characteristic 0 by props associated to operads
concentrated in degree 0: 

\begin{ex}[Fresse]\label{opex2}{\rm 
Let $k$ be a field of characteristic 0 and $\pp$ be an operad in $Vect_k$, the category of vector spaces over $k$. Let $\e_{\pp}$ denote its  associated prop
 as in
Example~\ref{opex1}.  
Then $\e_\pp$ is complete in $Vect_k$, i.e.~$\e_\pp(r,s)\cong\widehat\e_\pp(r,s):=\Hom((U_V)^{\ot r},(U_V)^{\ot s})$ for 
$U_V:\Fun^{\ot}(\e_{\pp},Vect_k)\to Vect_k$ the forgetful functor from $\pp$--algebras to vector spaces. Indeed, 
consider  the functors \\
\\
\begin{tabular}{rll}
$id^{\ot r}:$&$Vect_k\ \rar\ Vect_k $&$:\  X \mapsto X^{\ot r}$\\
$\pp(-)^{\ot s}:$&$Vect_k\ \rar\ Vect_k$ &$:\  X \mapsto \pp(X)^{\ot s} =\big(\bigoplus_n \pp(n)\ot_{\Si_n}X^{\ot n}\big)^{\ot s}.$
\end{tabular}\\
\\
Let $\Hom(id^{\ot r},\pp(-)^{\ot s})$ denote the natural transformations between these functors. 
There is map 
$$\e_{\pp}(r,s)=\!\!\!\!\!\bigoplus_{r_1+\dots+r_s=r}\!\!\!\pp(r_1)\ot \cdots \ot \pp(r_s) \ot_{\Si_{r_1}\x\cdots\x\Si_{r_s}}\Z\Si_{r}\ \
\sta{\eta}{\rar}\ \ \Hom(id^{\ot r},\pp(-)^{\ot s})$$ 
mapping a morphism $a=a_1\ot \cdots\ot a_s\ot \s$ 
to the natural transformation $\nu_a$ defined on a vector space 
$X$ as the composition
$$\nu_a(X):X^{\ot r}\sta{\s}{\rar} X^{\ot r}\inc (a_1\ot_{\Si_{r_1}} X^{\ot r_1})\ot\cdots\ot(a_s\ot_{\Si_{r_s}} X^{\ot r_s})\subset\pp(X)^{\ot s}$$
where the first map permutes the factors and the second is the inclusion determined by $a_1\ot \dots\ot a_s$.  
When $k$ is a field of characteristic 0, 
linearity implies that these are the only possible natural transformations from $id^{\ot r}$ to $\pp(-)^{\ot
  s}$, and hence that the map $\eta$ is an isomorphism (see \cite[App A, in particular (5.4)]{MacSymmetric}---the case
$s=1$ is stated in \cite[Prop 1.2.5]{Fre04}, see also \cite[Prop 2.3.12]{Fre09}). 
(Note that if we are not in characteristic 0, more maps are possible. For example the diagonal map $k\to k\ot_{\Si_2}k$ is linear in characteristic 2,
though not in characteristic 0.) 

There is an isomorphism $\Hom(id^{\ot r},\pp(-)^{\ot s})\cong \Hom((U_V)^{\ot r},(U_V)^{\ot s})$, which takes a natural transformation $\theta=(\theta_X)\in \Hom(id^{\ot r},\pp(-)^{\ot s})$  to the natural transformation given on a $\pp$--algebra $A$ by the map 
$$U(A)=A^{\ot r}\sta{\theta_{U(A)}}{\rar} \pp(A)^{\ot s}\sta{\ga_A^{\ot s}}{\rar} A^{\ot s}$$ for $\ga_A:\pp(A)\to A$ the structure map of $A$, and with inverse taking a natural transformation $\eta=(\eta_A)\in\Hom((U_V)^{\ot r},(U_V)^{\ot s})$  to the natural transformation defined on a vector space $X$ by 
$$X^{\ot r}\sta{i_X^{\ot r}}{\rar} \pp(X)^{\ot r}\sta{\eta_{\pp(X)}}{\rar} \pp(X)^{\ot s}$$
with $i_X:X\to id\ot X\subset \pp(X)$ the canonical inclusion.  
Now $ \Hom((U_V)^{\ot r},(U_V)^{\ot s})$ is  the completion of $\e_{\pp}$ as defined above.  
The composition of these two isomorphisms gives 
$$\e_{\pp}(r,s)\sta{\cong}{\rar} \Hom(id^{\ot r},\pp(-)^{\ot s})\cong \Hom((U_V)^{\ot r},(U_V)^{\ot s})=\widehat \e_{\pp}(r,s)$$
as claimed. 
}\end{ex}




%

\subsection{Example: the cap product}\label{capsec}

Given an algebra $A$ (i.e.~$A=\Phi(1)$ for $\Phi:\Ass\to \Comp$ a symmetric monoidal functor), 
we can consider its endomorphism prop $\End(A)$. It has morphisms from $p$ to $q$ defined as 
$$\End(A)(p,q):=\Hom(A^{\ot p},A^{\ot q})$$
where $A^{\ot p},A^{\ot q}$ are considered as objects of $\Comp$ and $\Hom$ here denotes the chain complex of morphisms in that category. 
The algebra structure of $A$ defines a symmetric monoidal functor $i:\Ai\to\End(A)$ (factoring through $\Ass$), 
so $\End(A)$ is a prop with $\Ai$--multiplication. 
The algebra $A$ is by construction an $\End(A)$--algebra: $\End(A)$ exactly encodes all the operations defined on $A$, and its associative 
algebra structure is the one obtained by pulling back its $\End(A)$--structure with $i$. 

Evaluating on the $\End(A)$--algebra $A$, one sees that  
the restriction map $$r:\Nat_{\End(A)}\rar \Nat_{\End(A)}^\ot$$ is injective. 
Explicitly,  $\Nat_{\End(A)}$  has morphism complexes
$$\Nat_{\End(A)}(\oc{m_1}{n_1},\oc{m_2}{n_2})\cong \!\!\!\!\!\!\prod_{j_1,\dots,j_{n_1}\ge 1}\ \bigoplus_{k_1,\dots,k_{n_2}\ge 1}\!\!\!\!\!\!\Hom(A^{\ot j_1+\dots+j_{n_1}+m_1},A^{\ot k_1+\dots+k_{n_2}+m_2})[k-j+n_1-n_2]$$
with differential $d_A+d_H+d^H$ where $d_H$ cyclically post-composes with multiplications and $d^H$ cyclically precomposes with multiplications. 
(Recall from Example~\ref{coex} that $d^H$ is {\em not} the Hochschild cochain differential!)

\medskip

Recall from \cite[XI.6]{CarEil56} the classical cap product in Hochschild homology: 
$$\cap\colon C_p(A,A)\ot C^q(A,A)\rar C_{p-q}(A,A).$$
Explicitly, given $a=a_0\ot\cdots\ot a_p\in C_p(A,A)$ and $D\in C^q(A,A)=\Hom(A^{\ot q},A)$, it is given by the formula
\begin{equation}\label{cap}
a\cap D=(-1)^{(|a|-|a_0|)|D|}a_0D(a_1,\dots,a_q)\ot a_{q+1}\ot\cdots\ot a_p
\end{equation}
(see e.g.~\cite[Sec 2]{NesTsy99}). This operation is constructed using the $\End(A)$--algebra structure of $A$, and is thus part of the
$\Nat_{\End(A)}$--action:

\begin{prop}\label{capprop}
There is an inclusion $F:C^*(A,A)\inc \Nat_{\End(A)}(\oc{0}{1},\oc{0}{1})$ satisfying that the composition
$$C_*(A,A)\ot C^*(A,A)\sta{id\ot F}{\inc} C_*(A,A)\ot \Nat_{\End(A)}(\oc{0}{1},\oc{0}{1}) \rar C_*(A,A)$$
 is the cap product. 
\end{prop}

\begin{proof}
Recall that $\Nat_{\End(A)}(\oc{0}{1},\oc{0}{1})\cong \prod_{j\ge 1}\bigoplus_{k\ge 1}\Hom(A^{\ot j},A^{\ot k})[k-j]$. Given a Hochschild cochain $D\in \Hom(A^q,A)=C^q(A,A)$,
define $F(D)$ to have $j$th component the map  $F(D)_j\in \Hom(A^{\ot j},A^{\ot j-q+1})$  defined by equation~(\ref{cap}) above with $p=j-1$, i.e.~ taking
$a=a_0\ot\cdots\ot a_{j-1}$ to $a\cap D$. 
We need to check that $F(dD)=d(F(D))$. For simplicity, we assume $A$ is concentrated in degree 0 and we leave the ``non-structural'' signs out. We have 
$(d(F(D)))_j=d_H(F(D)_{j})+d^H(F(D)_{j-1})$. Now  $$\big(d_H(F(D)_{j})\big)(a)=d(a\cap D)$$ is the Hochschild differential applied to the cap
product, and $$\big(d^H(F(D)_{j-1})\big)(a)=-D\cap da.$$ 
As $\big(F(dD)\big)(a)=dD\cap a$, the equation becomes 
$$dD\cap a = d(a\cap D) - D\cap da$$
for each $a=a_0\ot \cdots\ot a_{j-1}\in A^{\ot j}$ which is exactly saying that the cap product is a chain map. 
\end{proof}

Considering $C^*(A,A)$ as an algebra with the cup product,  Nest and Tsygan 
define an action of $C_*(C^*(A,A),C^*(A,A))$ on $C_*(A,A)$ extending the cap product. This is
further extended by Kontsevich-Soibelman in   \cite[Sec 11]{KS06} to an 
action of $B(n,1)\ot C^*(A,A)$ on $C_*(A,A)$, where $B(n,1)$ is a chain complex of certain configurations on a cylinder. 
In each case, the definition of the action, just as in the case of the cap product, gives  maps 
$$C_*(C^*(A,A),C^*(A,A))\rar\ \Nat_{\End(A)}(\oc{0}{1},\oc{0}{1})$$ and  $$B(n,1)\ot C^*(A,A)\rar\ \Nat_{\End(A)}(\oc{0}{1},\oc{0}{1})$$ 
which are chain maps exactly because the actions are compatible with the differential.

\section{Three computations of the formal operations}\label{compsec}

Let $\OO$ be the symmetric monoidal dg-category with objects the natural numbers and morphisms $\OO(n,m)$ the chain
complex of fat graphs with $n+m$ labeled leaves as described in Section~\ref{GraphApp}. (See Figure~\ref{Ofig} for an example of a morphism in $\OO$.)  Composition is defined by gluing leaves and the monoidal structure by disjoint union, just as for $\Ai$. In fact, the category $\Ai^+$ identifies as a subcategory of forests in $\OO$. 
\begin{figure}[h]
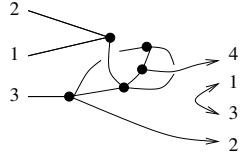

\begin{lpic}{OO(0.6,0.6)}
\end{lpic}
\caption{Morphism in $\OO(3,4)$ (with the outputs marked as outgoing arrows)}\label{Ofig}
\end{figure}
The chain complex $\OO(n,m)$ computes the homology of the moduli
space of Riemann cobordisms from $n$ to $m$ intervals, so that an $\OO$--algebra, i.e.~a strong symmetric monoidal functor $\Phi:\OO\to\Comp$, can be called an open topological 
conformal field theory (see  \cite{Pen04,Pen06,God07,costello07,Cos07} and \cite[Thm 2.2]{WahWes08}).  

Let $H_0(\OO)$ denote the dg-category with the same objects as $\OO$ but with morphisms  $H_0(\OO)(n,m):=H_0(\OO(n,m))$. Explicitly, by the above, $H_0(\OO)(n,m)$  is the chain complex
concentrated in degree 0 with one free generator for each topological type of cobordism from $n$ to $m$ intervals. A strong symmetric monoidal functor 
$\Phi:H_0(\OO)\to \Comp$, or $H_0(\OO)$--algebra, is called an open topological quantum field theory, and such field theories are in 1-1 correspondence with symmetric Frobenius algebras
(see \cite[Cor 4.5]{LauPfe}). Note that there is a quotient functor $\OO\to H_0(\OO)$.

 Including trees into all graphs defines a functor $i:\Ai^+\to\OO$, so $\OO$ and $H_0(\OO)$  are  props with unital $\Ai$--multiplication 
and open field theories have a well-defined Hochschild complex. 
In the present section, we identify the chain complexes of formal operations
$\bNat_\OO(\oc{m_1}{n_1},\oc{m_2}{n_2})$, $\bNat_{H_0(\OO)}(\oc{m_1}{n_1},\oc{m_2}{n_2})$ and $\bNat_{\Ai^+}(\oc{m_1}{n_1},\oc{m_2}{n_2})$,
 up to homotopy, with well-known objects. 

We start the section by analyzing the case of the category $\OO$. 
The cases $\Ai^+$ and $H_0(\OO)$ are then treated respectively as subcases and quotient of the case $\OO$. 

\subsection{Open topological conformal field theories}\label{OCsec}

Let $\OO$ be the open cobordism category described above and recall from Section~\ref{natsec} the chain complex
$\bNat_\OO(\oc{m_1}{n_1},\oc{m_2}{n_2})$ of formal operations on $\OO$-algebras, the hom complex of the reduced Hochschild functors
$$\bC^{(n_i,m_i)}:\Fun(\OO,\Comp)\to \Comp.$$ We will relate this complex 
to a corresponding morphism complex in the open-closed cobordism category $\OC$, which we describe now. 

\medskip

For the purpose of the present paper, we could define $\OC$ formally as a category with objects pairs of natural numbers $\oc{m}{n}$, with $m,n\ge 0$, and morphism complexes 
\begin{equation}\label{splittedOC}
\OC(\oc{m_1}{0},\oc{m_2}{n_2})\ \cong\ \bC^{n_2}(\OO(m_1,-))(m_2)=\!\!\!\!\bigoplus_{k_1,\dots,\,k_{n_2}}\!\!\OO(m_1,k+m_2)/_{U}\ot L_{k_1}\ot\cdots\ot L_{k_{n_2}},
\end{equation}
the reduced iterated Hochschild complex of $\OO(m_1,-)\colon \OO\to \Comp$ at $m_2$,  
and with $\OC(\oc{m_1}{n_1},\oc{m_2}{n_2})$ for $n_1>0$ a certain subcomplex of $\OC(\oc{m_1+n_1}{0},\oc{m_2}{n_2})$, together with a certain composition law. (Here $k=k_1+\dots+k_{n_1}$ and $L_{k_i}$ is the free module on one generator $l_{k_i}$ 
in degree $k_i-1$ of Section~\ref{sec1}.)  
The fact that it can be described this way is a major ingredient in our identification of the natural operations, but it would difficult to describe which subcomplex   $\OC(\oc{m_1}{n_1},\oc{m_2}{n_2})$ really is, or define the composition, without introducing open-closed cobordisms and black and white graphs, which we do now. 
At the same time, we will give geometric meaning to the category $\OC$. 

\medskip

An open-closed cobordism is a surface $S$ whose boundary has an incoming and outgoing part: $\del S\supset\del_{in}S\sqcup\del_{out}S$, 
where $\del_{in}S$ and $\del_{out}S$ are disjoint unions of labeled circles
and closed intervals in $\del S$, with $S$  then considered as a cobordism from $\del_{in}S$ to $\del_{out}S$. 
The part of the boundary which is neither incoming nor outgoing is called {\em free}:  $\del_{free}S=\del S\minus (\del_{in}S\cup \del_{out}S)$. 
The open-closed cobordism category $\OC$ considered here has objects pairs of natural numbers $\oc{m}{n}$, thought of as a union of $m$ intervals and $n$ circles, and
morphisms $\OC(\oc{m_1}{n_1},\oc{m_2}{n_2})$ a chain complex whose homology is that of the moduli space of Riemann cobordisms from $m_1$
intervals and $n_1$ circles to $m_2$ intervals and $n_2$ circles satisfying that each component has non-empty free or incoming boundary. 

To describe $\OC$ explicitly, we use the language of black and white graphs. A {\em black and white graph} is a fat graph whose set of vertices is given as $V=V_b\sqcup V_w$, with $V_b$ the set of {\em black} vertices and $V_w$ the set of {\em white} vertices. Black vertices are assumed to be at least trivalent whereas white vertices are allowed to be any valence at least 1. In addition, white vertices are labeled, and come equipped with the data of a start half-edge at the vertex. 
The graph $l_n$ from Section~\ref{sec1} is a black and white graph. 
Figure~\ref{bwgraph}(a) gives a more generic example. Figure~\ref{bwgraph}(b) shows how to associate a surface with boundaries to a black and white graph, fattening the graph and inserting boundary circles where there were white vertices. A {\em boundary cycle} in the graph is a sequence of half-edges in the graph corresponding to a boundary component of that surface not coming from a white vertex. 
\begin{figure}[h]
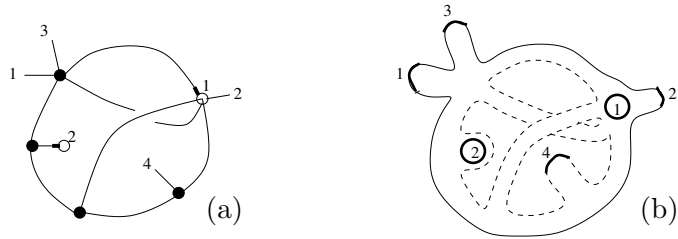

\begin{lpic}{bwgraph.3(0.52,0.52)}
\lbl[b]{55,3;(a)}
\lbl[b]{165,3;(b)}
\end{lpic}
\caption{Black and white graph and associated cobordism. 
}\label{bwgraph}
\end{figure}

Black and white graphs form a chain complex just like fat graphs, defining the degree of a graph as $\sum_{v\in V_b}(|v|-3) + \sum_{v\in V_w}(|v|-1)$, and with the differential given as sums of all possible blow-ups, a white vertex blowing up as a pair of a white and a black vertex. In fact, one can think of a black and white graph as made out of one copy of the graph $l_{|v|}$ of Section~\ref{sec1} for each white vertex $v$, attached along a fat graph, and the differential at such a white vertex $v$ is that of $l_{|v|}$---This translates to the formal description of $\OC$ given in equation (\ref{splittedOC}) above.  
We have also mentioned this correspondence earlier in  
Remark~\ref{LLgraphs} when describing the functor $\LL$ used to define the Hochschild complex. 
We refer to  \cite[Sec 2.3,2.5]{WahWes08} for more details about the chain complex of black and white
graphs. 

Now $\OC(\oc{m_1}{n_1},\oc{m_2}{n_2})$ is a chain complex 
generated by black and white graphs with $n_2$ white vertices
modeling the $n_2$ outgoing closed boundary components, 
and $n_1+m_1+m_2$ labeled leaves such that the first $n_1$ leaves each are sole leaves in their boundary cycle,  
 representing the $n_1$ closed incoming, $m_1$ open incoming and $m_2$ open outgoing boundaries. 
Composition is obtained by summing over all ways of removing  the white vertices of the first graph (its outgoing closed boundaries) and attaching their leaves to the corresponding incoming boundary cycle of the second graph, in a way that respects the cyclic ordering and preserves the grading, and finally attaching in pairs the leaves modeling open boundaries as in $\OO$. 
We refer to 
 \cite[Sec 2.8]{WahWes08} for further details about the category $\OC$, including a precise definition of the composition.
It follows from \cite[Prop 6.1.3]{costello07} that this chain complex computes the homology of the moduli space of Riemann cobordisms from $m_1$
intervals and $n_1$ circles to $m_2$ intervals and $n_2$ circles satisfying that each component has at least an incoming or a free boundary (see
\cite[Thm 5.4]{WahWes08}). The fact that composition in $\OC$ corresponds to gluing of surfaces on the moduli space is shown in \cite{Ega14A}.

\medskip

Cutting around the white vertices yields  the isomorphism 
$$\OC(\oc{m_1}{0},\oc{m_2}{n_2})\ \cong\ \bC^{n_2}(\OO(m_1,-))(m_2)=\!\!\!\!\bigoplus_{k_1,\dots,k_{n_2}}\!\!\OO(m_1,k+m_2)/_{U}\ot L_{k_1}\ot\cdots\ot L_{k_{n_2}}$$
mentioned above. (This is the content of \cite[Lem 6.1]{WahWes08}.)
In the language of \cite{WahWes08}, this means that $\OC$ is an extension of the Hochschild core category of $\OO$.
Corollary~5.12 of \cite{WahWes08} can then be rephrased as saying that $\OC$ acts naturally
on the functors $$\bC^n(-)(m):\Fun(\OO,\Comp)\rar \Comp$$
i.e.~that 
there is a functor $$J_\OO:\OC\to \overline{\Nat}_\OO.$$ To describe this functor explicitly
on morphism complexes, we first rewrite its target using the identifications from Theorem~\ref{natural} and \cite[Lem.~6.1]{WahWes08}:  
$$\bNat_\OO(\oc{m_1}{n_1},\oc{m_2}{n_2})\cong \bD^{n_1}\bC^{n_2}\OO(-,-)(m_1,m_2)\cong \bD^{n_1}(\OC(\oc{-}{0},\oc{m_2}{n_2}))(m_1)$$
which exhibits $\bNat_\OO(\oc{m_1}{n_1},\oc{m_2}{n_2})$ as a subcomplex of the product $\prod_{\uj}\OC(\oc{m_1+j}{0},\oc{m_2}{n_2})[n_1-j]$ over $\uj=(j_1,\dots,j_{n_1})$
and with  $j=j_1+\dots+j_{n_1}$ as before. 
Then the map 
$$J_{\OO}:\ \OC(\oc{m_1}{n_1},\oc{m_2}{n_2})\rar \bNat_\OO(\oc{m_1}{n_1},\oc{m_2}{n_2})\subset \prod_{\uj}\OC(\oc{m_1+j}{0},\oc{m_2}{n_2})[n_1-j]$$
can be explicitly identified as the map that   
takes a graph $G$ to the sequence $\{G\circ (l_{\uj}+ id_{m_1})\}_{\uj}$, where 
$l_{\uj}=l_{j_1}\ot\cdots\ot l_{j_{n_1}}\in L_{j_1}\ot\cdots\ot L_{j_{n_1}}$ corresponds in $\OC(\oc{j}{0},\oc{0}{n_1})$ to the black and white graph defined by the disjoint union of the graphs $l_{j_i}$,
with $\circ$ the composition in $\OC$. 
More explicitly, 
the $\uj$th component of the image of $G$ is the sum of all graphs that can be 
obtained from $G$ by gluing $l_{j_i}-1$ labeled leaves cyclically at the 
vertices of the boundary cycle of $G$ corresponding to its $i$th incoming closed boundary for each $i=1,\dots,n_1$. 
One can check that this map has image in the reduced coHochschild complex.

Note that the map $J_{\OO}$ is injective on each morphism complex. A splitting can be constructed using  
the map $ \bNat_\OO(\oc{m_1}{n_1},\oc{m_2}{n_2})\to \OC(\oc{m_1+n_1}{0},\oc{m_2}{n_2}) $ associating to a sequence $\{G_{\uj}\}_{\uj}\in\prod_{\uj}\OC(\oc{m_1+j}{0},\oc{m_2}{n_2})[n_1-j]$ its $\uj$th component for $\uj=(1,\dots,1)$. 
The categories are however far from being isomorphic. In particular, the morphisms complexes in $\bNat_{\OO}$, unlike in $\OC$, are not positively graded. 
We will however show the following result:

\begin{thm}\label{OC}
The functor $J_\OO:\OC\sta{\sim}{\inc}\bNat_\OO$ is a quasi-isomorphism of categories, which is split injective on each morphism complex.  
\end{thm}

The main ingredient in the proof is a description of  the complex $\bD^{n_1}(\OC(\oc{-}{0},\oc{m_2}{n_2}))(m_1)$ in terms of a cosimplicial set of partitions, 
which we define first. 

Let $X$ be an oriented 1-manifold without boundary (i.e.~a collection of circles and open intervals). 
Let $\CPart(q,X)=\pi_0\operatorname{Conf}(q,X)$ denote the components of the configuration
space of $q$ ordered points in $X$. These can be thought of as partitions of $\{1,\dots,q\}$ which are ordered on the interval
components of $X$ and cyclically ordered on the circle components. 

\begin{lem}\label{CPart}
For each oriented 1-manifold $X$ without boundary, the sets $\{\CPart(q+1,X)\}_{q\ge 0}$ form a cosimplicial set. 
\end{lem}

\begin{proof}
Let $(p_0,\dots,p_q)\in X^{q+1}$ represent an element in $\CPart(q+1,X)$. 
Define $$d^i:\CPart(q+1,X)\to \CPart(q+2,X)$$ by 
adding a copy $\tilde p_i$ of $p_i$ to the right of $p_i$ for $0\le i\le q$, and by adding a copy $\hat p_0$ of $p_0$ to the left
of $p_0$ for $i=q+1$, to get a new configuration represented by 
$(p_0,\dots,p_i,\tilde p_i,\dots,p_q)$ (resp. $(p_0,\dots,p_q,\hat p_0)$). 
Define also  
$$s^i:\CPart(q+1,X)\to \CPart(q,X)$$
forgetting $p_{i+1}$, $0\le i\le q-1$, and relabeling respecting the order. 

We are left to check the cosimplicial identities:
$$\begin{array}{lcll}
d^jd^i&=&d^id^{j-1}&i<j\\
s^jd^i&=&d^is^{j-1}& i<j\\
s^jd^j&=&1=s^jd^{j+1}& \\
s^jd^i&=&d^{i-1}s^j& i>j+1\\
s^js^i&=&s^is^{j+1}& i\le j
\end{array}$$ 
For the first identity, if $j\le q+1$, both compositions repeat $p_i$ and $p_{j-1}$ to their
right and if $j=q+2$, both compositions repeat $p_i$ and $p_0$, one to the right, the other one to the left, or they repeat
twice $p_0$ to its left if $i=q+1$. 
For the second, both compositions double $p_i$ to its right and forget $p_j$. For the third, the first composition adds a copy of $p_j$
and forgets the added point, while the second composition adds a copy of $p_{j+1}$ and forget the original $p_{j+1}$ if $j<q$, or, when $j=q$, it
adds a copy of $p_0$ to its right and forgets it. For the fourth equality, both compositions forget $p_{j+1}$ and add $p_i$ to its
right (if $i\le q$) or $p_0$ to its left (if $i=q+1$). Finally in the last one, both compositions forget $p_{i+1}$ and $p_{j+2}$.  
\end{proof}

There is a chain complex $\oplus_q\Z \CPart(q+1,X)$ associated to the cosimplicial set 
$\CPart(\bullet+1,X)$, with $\Z \CPart(q+1,X)$ in degree $-q$, and with differential the alternating
sum of the coface maps  
$$d_K=\sum (-1)^i d^i:\Z \CPart(q,X)\to \Z \CPart(q+1,X)$$ 
raising the degree by one. The appendix Section~\ref{cosimpsec} studies the properties of such chain complexes. In particular,  
Proposition~\ref{cosimpcor} says that such a chain complex has homology concentrated in degree 0. 

Let $\CPart_c(q,X)$ denote the subset of $\CPart(q,X)$ of {\em complete configurations}, those with at least one point in each circle 
component of $X$. 
Note that the coface maps $d^i$ restrict to the subsets $\CPart_c(q,X)$. In fact, for $x\in \CPart(q,X)$ with $d_K(x)\neq 0$, we have 
$d_K(x)\in \Z\CPart_c(q+1,X)$ if and only if 
$x\in \CPart_c(q,X)$, so that  the subcomplex 
$(\Z \CPart_c(*,X),d_K)$ splits off as a union of components of the chain complex $(\Z \CPart(*,X),d_K)$.
We reinterpret $(\Z \CPart_c(*,X),d_K)$ as the coHochschild complex of a functor $\Z \CPart_c(-,X)$ as follows: 
with $\Ass$ denoting the prop of (non-unital) associative algebras, define
$$\Z \CPart_c(-,X): (\Ass)^{op} \rar \Comp$$
on objects by $\Z \CPart_c(-,X)(n)=\Z \CPart_c(n,X)$  concentrated in degree 0,  and on morphisms 
by doubling and relabeling points in the configurations according to the morphisms in $\Ass$. 
As $\Z \CPart_c(-,X)(n)$  is concentrated in degree 0 for each $n$, we have that 
its coHochschild complex is really a direct sum complex, with $\Z \CPart_c(n,X)$ in degree $1-n$. 
In fact we have the following: 

\begin{lem}\label{ZK} For any oriented 1-manifold without boundary $X$, we have  an isomorphism    
$(\Z \CPart_c(*,X),d_K)\cong D(\Z \CPart_c(-,X))(0)$ between the chain complex of the cosimplicial set of complete configurations in $X$, and the coHochschild complex
of the associated functor. 
\end{lem}



For each $n_1,n_2,m_1,m_2$, we consider the functor 
$$\OC(\oc{m_1+-}{n_1},\oc{m_2}{n_2}):\ \OO^{op} \rar \Comp$$
taking the value $\OC(\oc{m_1+q}{n_1},\oc{m_2}{n_2})$ on the object $q$ in $\OO$, and defined on morphisms by precomposition in $\OC$ 
via the inclusion of $\OO$ in $\OC(\oc{m_1+-}{n_1},\oc{m_1+-}{n_1})$ tensoring  with the 
identity on $\oc{m_1}{n_1}$.  

We give now a model for the restriction of $\OC(\oc{m_1+-}{n_1},\oc{m_2}{n_2})$ to $\Ai^{op}$ which splits out the configuration information defined by the
incoming and outgoing boundary components of a cobordism: 
Let $$(\OC\!\ot\! \CPart)(\oc{m_1}{n_1},\oc{m_2}{n_2}):\ (\Ass)^{op} \rar \Comp$$ be defined on the object $q$ in $\Ass$ by 
$$(\OC\!\ot\! \CPart)(\oc{m_1}{n_1},\oc{m_2}{n_2})(q)=\bigoplus_{b=0}^{q}\bigoplus_S \OC_S(\oc{m_1}{n_1+b},\oc{m_2}{n_2})\ot \Z\CPart_c(q,\textstyle{\coprod}_b S^1 \cup \del_IS) $$
where $S$ runs over the components of $\OC(\oc{m_1}{n_1+b},\oc{m_2}{n_2})$, and $\OC_S$ denotes the corresponding component, 
and $\del_IS$ denotes the interval components of the free boundary of $S$. 
The differential is the differential in $\OC$.  On morphisms, $\Ass$ acts via its action on the second factor. 

We can pull back the above two functors via $i:\Ai\to\OO$ and $j:\Ai\to \Ass$ respectively to functors defined on $(\Ai)^{op}$. 

\begin{lem}\label{OCK}
There is a quasi-isomorphism
of functors $$i^*\OC(\oc{m_1+-}{n_1},\oc{m_2}{n_2})\arsim j^*(\OC\!\ot\! \CPart)(\oc{m_1}{n_1},\oc{m_2}{n_2}):(\Ai)^{op} \rar \Comp.$$
\end{lem}

\begin{proof}
We first define for each $q$ a map  
$$\beta_q: \OC(\oc{m_1+q}{n_1},\oc{m_2}{n_2})\rar \bigoplus_{b=0}^{q} \bigoplus_S\OC_S(\oc{m_1}{n_1+b},\oc{m_2}{n_2}) \ot \CPart_c(q,\textstyle{\coprod}_b S^1 \cup \del_IS).$$
 Recall that a closed incoming boundary component is modeled in $\OC$ by a leaf alone in its boundary cycle in the graph. 
A graph $G\in \OC(\oc{m_1+q}{n_1},\oc{m_2}{n_2})$ is a graph with $n_2$ white vertices, $q$ leaves $p_1,\dots,p_q$ and $n_1+m_1+m_2$ leaves $\la_1,\dots,\la_{n_1+m_1+m_2}$, with 
the first $n_1$ leaves alone in their boundary cycle. Let $b$ be the number of boundary cycle of $G$ having leaves only of type $p_i$. Define $\be_q(G)=G'\ot P$, 
where $G'$ is obtained from $G$ by forgetting the leaves $p_1,\dots,p_q$ (and their attaching trees) except for $p_{i_1},\dots,p_{i_b}$, the first occurring leaf of that type in each boundary cycle of $G$ 
with only leaves labeled by 
$p_i$'s, provided that $G$ and $G'$ have the same degree, i.e.~provided that the forgotten leaves were attached at trivalent black vertices or univalent white vertices. 
Otherwise we set $\be_q(G)=0$. 
Let $S$ be the topological cobordism associated to $G'$. 
The second component, $P$, remembers the configuration defined by the leaves $p_i$ in $\del S$. This configuration is supported in $\coprod_bS^1\cup\del_IS$. 
(See Figure~\ref{configsplit} for an example.)
\begin{figure}[h]
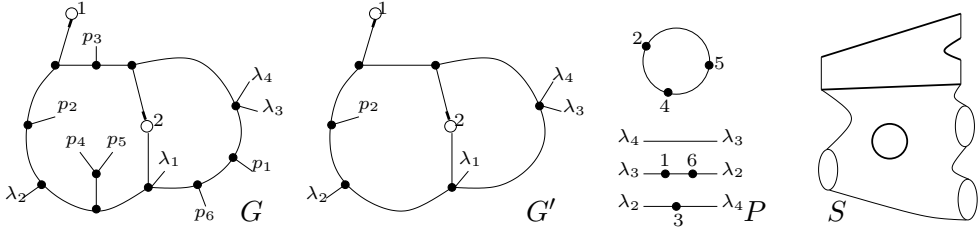

\begin{lpic}{configsplit-eps-converted-to(0.36,0.36)}
\lbl[b]{84,0;$G$}
\lbl[b]{190,0;$G'$}
\lbl[b]{268,0;$P$}
\lbl[b]{298,0;$S$}
\lbl[b]{17,40;{\tiny $p_2$}}
\lbl[b]{20,27;{\tiny $p_4$}}
\lbl[b]{35,27;{\tiny $p_5$}}
\lbl[b]{-2,7;{\tiny $\la_2$}}
\lbl[b]{26,65;{\tiny $p_3$}}
\lbl[b]{54,20;{\tiny $\la_1$}}
\lbl[b]{88,51;{\tiny $\la_4$}}
\lbl[b]{92,39;{\tiny $\la_3$}}
\lbl[b]{88,18;{\tiny $p_1$}}
\lbl[b]{67,0;{\tiny $p_6$}}
\lbl[b]{22,76;{\tiny $1$}}
\lbl[b]{50,35;{\tiny $2$}}
\lbl[b]{127,40;{\tiny $p_2$}}
\lbl[b]{108,7;{\tiny $\la_2$}}
\lbl[b]{164,20;{\tiny $\la_1$}}
\lbl[b]{198,51;{\tiny $\la_4$}}
\lbl[b]{202,39;{\tiny $\la_3$}}
\lbl[b]{132,76;{\tiny $1$}}
\lbl[b]{160,35;{\tiny $2$}}
\lbl[b]{225,65;{\tiny $2$}}
\lbl[b]{235,40;{\tiny $4$}}
\lbl[b]{255,56;{\tiny $5$}}
\lbl[b]{235,20;{\tiny $1$}}
\lbl[b]{245,20;{\tiny $6$}}
\lbl[b]{240,-2;{\tiny $3$}}
\lbl[b]{222,28;{\tiny $\la_4$}}
\lbl[b]{260,28;{\tiny $\la_3$}}
\lbl[b]{222,16;{\tiny $\la_3$}}
\lbl[b]{260,16;{\tiny $\la_2$}}
\lbl[b]{222,4;{\tiny $\la_2$}}
\lbl[b]{260,4;{\tiny $\la_4$}}
\end{lpic}
\caption{Graph $G$ with $n_1=1$, $n_2=2$, $m_1=1$, $m_2=2$ and $q=6$, and its image $G'\ot P$ under $\beta_6$, with $S$ the underlying cobordism type of $G'$, where the thicker lines show 
the support of $P$.}\label{configsplit}
\end{figure}

We check that $\be_q$ is a chain map, i.e. that it commutes with the differential in $\OC$.  
If $\be_q(G)=G'\ot P$ with $G'\neq 0$, then $\be_q(d_{\OC}G)=d_{\OC}G'\ot P$  as the forgotten leaves where attached at trivalent black vertices or univalent white vertices in that case. 
On the other hand, if $\be_q(G)=0$, then each component of $d_{\OC}G$ will also be mapped to 0 by $\be_q$, unless 
$G$ has a single valence 4 black vertex or valence 2 white vertex with a $p_i$ forgotten by $\beta_q$. 
In this case, there are two terms in $d_{\OC}G$ on which $\be_q$ does not vanish, but their images will be isomorphic and with opposite orientations. 
Hence each $\be_q$ is a chain map. 

Furthermore, the $\be_q$'s assemble to a natural transformation of functors: for $f\in(\Ai)^{op}$,  $\be(f(G))=j(f)\be(G)$ as the image of a graph $G$
under positive degree morphisms of $\Ai$
are killed by $\be$, and the action of degree 0 morphisms on graphs corresponds exactly to the action on $\CPart_c$ via the projection $j:\Ai\to\Ass$. 

\smallskip

We are left to show that each $\be_q$ is a quasi-isomorphism.  
Note first that the above map is an isomorphism on sets of components as there is exactly one topological type of cobordism on the left hand side corresponding to each pair $(S,P)$ on the 
right hand side. 


Finally, the homology groups on the left hand side only depend on $b$, not on the particular 
partition in $\CPart_c(\bullet+1,\del_IS\cup \del_{[b]}S)$ as the 
mapping class group of a surface fixing any number of points in a boundary component is isomorphic to the mapping class group fixing the whole boundary. 
\end{proof}

We are now ready to prove our second lemma, which, iterated $n_1$ times, will prove the theorem. 

\begin{lem}\label{OOn1}
There is a quasi-isomorphism
of functors 
$$\OC(\oc{-}{n_1+1},\oc{m_2}{n_2})\arsim D(\OC(\oc{-}{n_1},\oc{m_2}{n_2}))(-):\ \OO^{op}\rar \Comp$$
defined at the object $m_1$ by taking a graph $G$ to $\{G\circ (l_j+id_{m_1})\}_{j\ge 1}$, 
using the identification $D(\OC(\oc{-}{n_1},\oc{m_2}{n_2}))(m_1)\cong\prod_{j\ge 1}\OC(\oc{m_1+j}{n_1},\oc{m_2}{n_2})[1-j]$. 
\end{lem}

(The map does actually have image in the quasi-isomorphic reduced subcomplex.)

\begin{proof}
Note first that the map is indeed a natural transformation as $\OO$ acts by composition on boundaries that are not affected by the glued $l_j$'s. 

We are left to check that it induces isomorphisms in homology pointwise. 
We have $D(\OC(\oc{-}{n_1},\oc{m_2}{n_2}))(m_1)\cong D(\OC(\oc{m_1+-}{n_1},\oc{m_2}{n_2}))(0)$. 
 By Lemma~\ref{OCK} and the homotopy invariance of the coHochschild construction (Proposition~\ref{htpy}), we thus get 
$$D(\OC(\oc{-}{n_1},\oc{m_2}{n_2}))(m_1)\simeq D((\OC\!\ot\! \CPart)(\oc{m_1}{n_1},\oc{m_2}{n_2}))(0).$$
Recall from Section~\ref{htpysec} the filtration
$$F^s=\prod_{q\ge s}\bigoplus_{b=0}^{q}\bigoplus_S \OC_S(\oc{m_1}{n_1+b},\oc{m_2}{n_2})\ot
\Z\CPart_c(q,\textstyle{\coprod}_b S^1 \cup \del_IS)$$
which is exhaustive and complete.
 Its $E^1$-term has 
$$E^1_{-p,q}=\bigoplus_{b=0}^{p+1}\bigoplus_S H_q(\OC_S(\oc{m_1}{n_1+b},\oc{m_2}{n_2}))\ot \Z\CPart_c(p+1,\textstyle{\coprod}_b S^1 \cup \del_IS)$$
and the $d^1$-differential is the coHochschild differential, which, by Lemma~\ref{ZK},  can be identified with the differential of the  
chain complex associated to the cosimplicial set $\CPart(\bullet,\textstyle{\coprod}_b S^1 \cup \del_IS)$, restricted to the split subcomplex $\Z\CPart_c$. 

By Proposition~\ref{cosimpcor}, 
the homology of this complex is concentrated in degree $0$, with a summand for each element $x$ in $\CPart_c(1,\textstyle{\coprod}_b S^1 \cup \del_IS)$ 
satisfying that $d^0x=d^1x$. Such an $x$ is a configuration of a single element $p_0$, and $d^0,d^1$ double it as $(p_0,p_1)$ and
$(p_1,p_0)$. These two configurations are in the same component exactly when $p_0$ is alone on its boundary component. There is exactly one such 
configuration in $\CPart_c(1,\textstyle{\coprod}_b S^1 \cup \del_IS)$ when $b=1$ and none if $b\neq 1$. Hence the $E^2$-term is given as $E^2_{-p,q}=0$ for $p>0$ and 
$$E^2_{0,q}=H_q(\OC(\oc{m_1}{n_1+1},\oc{m_2}{n_2})).$$
Now using the Eilenberg-Moore comparison theorem (\cite[Thm 5.5.11]{WeiHomAlg}) for the trivial filtration 
of $H_*(\OC(\oc{m_1}{n_1+1},\oc{m_2}{n_2}))$ given by  $F^1= H_*(\OC(\oc{m_1}{n_1+1},\oc{m_2}{n_2}))$, $F^s=0$ for $s>1$, the map in the statement of the corollary gives an isomorphism on the $E^2$-pages, and hence an isomorphism
$ H_*(\OC(\oc{m_1}{n_1+1},\oc{m_2}{n_2}))\sta{\cong}{\rar}H_*(D((\OC\!\ot\! \CPart)(\oc{-}{n_1},\oc{m_2}{n_2}))(m_1))$.
\end{proof}

\begin{proof}[Proof of Theorem~\ref{OC}]
Applying Lemma~\ref{OOn1} $n_1$ times together with the homotopy invariance of $D$ (Proposition~\ref{htpy}), we get a sequence of quasi-isomorphisms
$$\OC(\oc{m_1}{n_1},\oc{m_2}{n_2})\arsim D(\OC(\oc{-}{n_1-1},\oc{m_2}{n_2}))(m_1)\arsim\,\cdots\,\arsim D^{n_1}(\OC(\oc{-}{0},\oc{m_2}{n_2}))(m_1)$$
whose composition is the map given in the statement of the theorem. As that map has image in the subcomplex 
$\bNat_{\OO}(\OC(\oc{m_1}{n_1},\oc{m_2}{n_2}))=\bD^{n_1}(\OC(\oc{-}{0},\oc{m_2}{n_2}))(m_1)$, 
which is quasi-isomorphic to $D^{n_1}(\OC(\oc{-}{0},\oc{m_2}{n_2}))(m_1)$ by Proposition~\ref{conormalization}, we get the desired quasi-isomorphism. 
\end{proof}

\begin{rem}{\rm 
There is a space-level version of the above algebraic phenomenon. Given a surface $S$ with $n+m$ boundary components $\del_n S\cup
\del_m S$, with $m\ge 1$, there is a cosimplicial space $\mathcal{K}^\bullet(S,n)$ with 
$$\mathcal{K}^q(S,n):=\coprod_{p_0,\dots,p_k\in \del_mS} B\Ga(S;\del_n S\cup \{p_0,\dots,p_k\})$$
where $\Ga(S;\del_n S\cup \{p_0,\dots,p_k\})$ denotes the mapping class group of $S$ fixing the first $n$ boundary components of $S$
pointwise 
as well as fixing each of the marked points $p_0,\dots,p_k$ lying in the last $m$ boundary components of $S$. 
The coface and codegeneracy maps  are defined as for the cosimplicial set $\CPart(\bullet,X)$ by doubling and forgetting marked
points. One can then check that $Tot(\mathcal{K}^\bullet(S,n))\cong B\Ga(S;\del_{n+1} S)$. 
}\end{rem}

\subsection{Unital $\Ai$-algebras}

We want to restrict the argument of the previous section for the open cobordism category $\OO$ to its subcategory $\Ai^+$. 
Recall from \cite[Prop 6.12]{WahWes08} that 
$$\bC^{n_2}(\Ai^+(m_1,-)(m_2)\cong \bigoplus_A \OC_A(\oc{m_1}{0},\oc{m_2}{n_2})$$
is a sum of components of $\OC$, 
where $A$ runs over the cobordisms which are a union of $m_2$ discs, each with exactly one open outgoing boundary, and $n_2$ annuli, each with
precisely one closed outgoing boundary, and with a total of $m_1$ incoming open boundaries distributed on the free boundaries of the discs and
annuli. 
In particular, we can now identify 
$$\bNat_{\Ai^+}(\oc{m_1}{n_1},\oc{m_2}{n_2})\cong \bD^{n_1}(\bigoplus_A\OC_A(\oc{-}{0},\oc{m_2}{n_2}))(m_1)$$
which exhibits $\bNat_\OO(\oc{m_1}{n_1},\oc{m_2}{n_2})$ as a subcomplex of $\prod_{j_1,\dots,j_{n_1}}\!\!\bigoplus_A\OC_A(\oc{m_1+j}{0},\oc{m_2}{n_2})[n_1-j]$. 

Let $\Ann$ denote the subcategory of $\OC$ with the same objects, with 
$$\Ann(\oc{m_1}{n_1},\oc{m_2}{n_2})\cong \bigoplus_{A'} \OC_{A'}(\oc{m_1}{n_1},\oc{m_2}{n_2})$$
where $\A'$ runs over cobordisms which are a union of $m_2$ discs and $n_2-n_1$ annuli as before, 
union with $n_1$ annuli each with one incoming closed and one outgoing closed
boundary components. This chain complex can thus also be identified with
$\bigoplus_A \OC_A(\oc{m_1}{0},\oc{m_2}{n_2-n_1})\ot\OC_{S^1\x I}(\oc{0}{1},\oc{0}{1})^{\ot n_1}$, with $A$ as above. 
Restriction of the functor $J_\OO$ of the previous section of $\Ann$ yields a commutative diagram: 
$$\xymatrix{\Ann(\oc{m_1}{n_1},\oc{m_2}{n_2})\ar[r]^-{J_{\!\Ai^+}} \ar@{^(->}[d]& \bNat_{\Ai^+}(\oc{m_1}{n_1},\oc{m_2}{n_2}) \ar@{^(->}[d]\ar@{^(->}[rr] && 
  \prod_{\uj}\bigoplus_{A}\OC_A(\oc{m_1+j}{0},\oc{m_2}{n_2})[n_1-j] \ar@{^(->}[d]\\
\OO(\oc{m_1}{n_1},\oc{m_2}{n_2})\ar[r]^-{J_\OO}& \bNat_{\OO}(\oc{m_1}{n_1},\oc{m_2}{n_2})\ar@{^(->}[rr] &&  
  \prod_{\uj}\OC(\oc{m_1+j}{0},\oc{m_2}{n_2})[n_1-j]
}$$  
where $J_{\Ai^+}$, the restriction of $J_\OO$ to $\Ann$, 
takes a graph $G$ to the sequence $\{G\circ (l_{\uj}+id_{m_1})\}_{\uj}$, where $G\circ (l_{\uj}+id_{m_1})$ is obtained from $G$ by replacing its $i$th incoming closed
boundary by $j_i$ cyclically ordered 
incoming open boundaries for each $0\le i\le n_1$. 

The proof of Theorem~\ref{OC} restricts to the annuli components arising above to prove that $J_{\Ai^+}$ is also a quasi-isomorphism:  

\begin{thm}\label{Ai}
 The functor $J_{\Ai^+}:\Ann\to\bNat_{\Ai^+}$ is a quasi-isomorphism of categories. 
\end{thm}

Recall from \cite[Prop 6.14]{WahWes08} that $H_*(\Ann)$ is generated as a symmetric monoidal category by the unit and multiplication of the algebra, 
the $\De$-operator representing 
Connes-Rinehart's $B$ operator on the Hochschild complex of an associative algebra, and the inclusion of the algebra in its Hochschild complex.



\subsection{Strict Frobenius algebras}\label{SDsec}

Let $H_0(\OO)$ be the dg-category obtained from $\OO$ by taking its 0th homology, as described at the beginning of the section, whose algebras are symmetric Frobenius algebras.  
We will describe the formal operations on $H_0(\OO)$--algebras in terms of Sullivan diagrams. 

Following \cite{WahWes08}, we define here {\em Sullivan diagrams} as equivalence classes of
black and white graphs having only trivalent black vertices, where the equivalence relation is generated by  the boundary of graphs with a single  valence 4 black
vertex.  By \cite[Thm 2.6]{WahWes08}, this is equivalent to the more classical definition, which, loosely speaking, defines Sullivan diagrams 
 as equivalence classes of fat graphs build from a number of circles by attaching chords, where the chords are
thought of as having length 0. (Figure~\ref{tgmg} shows examples of a few Sullivan diagrams both represented as black and white graphs and as ``classical''
Sullivan diagrams. The white vertices in the black and white graphs description correspond to the circles in the classical picture.) Sullivan diagrams model the unimodular harmonic compactification of moduli space \cite{EgaKup}. 

The category $\SD$ of Sullivan diagram can then be directly defined as the
quotient category of the open-closed category $\OC$ by the graphs with higher valence black vertices, and by the boundaries of such. (See \cite[Sec 2.10]{WahWes08} for a less concise description of that category.)

Lemma 6.6 of \cite{WahWes08} shows that, cutting around the white vertices gives an isomorphism  
$$\SD(\oc{m_1}{0},\oc{m_2}{n_2})\ \cong\ \bC^{n_2}\big(H_0(\OO)(m_1,-)\big)(m_2)$$
which says, in the language of \cite{WahWes08}, that $\SD$ is an extension of the Hochschild core category of $H_0(\OO)$. 
Hence by \cite[Cor 5.12]{WahWes08}, $\SD$ acts naturally on the functors $\bC^{(n,m)}:\Fun(H_0(\OO),\Comp)\to\Comp$, that is  
there is a functor
$$J_{H_0}:\SD \rar \bNat_{H_0(\OO)}$$
which is an inclusion on morphism complexes just as in the case of the prop $\OO$. 
We prove here that this functor is a quasi-isomorphism: 

\begin{thm}\label{H0OC}
The functor $J_{H_0}:\SD\sta{\sim}{\inc}\Nat_{H_0(\OO)}$ is a quasi-isomorphism of categories, which is split-injective on each morphism complex. 
\end{thm}

We will prove this theorem by adapting the proof of Theorem~\ref{OC} to the present case. We first need the analogue of Lemma~\ref{OCK}. 

\smallskip

Note that the components of the morphism complexes of Sullivan diagrams are in one to one correspondence with those of $\OC$ as the equivalence
relation defining Sullivan diagrams respects the components. 

Just as in the case of the category $\OC$, let 
$$(\SD\!\ot\! K)(\oc{m_1}{n_1},\oc{m_2}{n_2}):\ (\Ass)^{op} \rar \Comp$$ be defined on objects by 
$$(\SD\!\ot\! K)(\oc{m_1}{n_1},\oc{m_2}{n_2})(q)=\bigoplus_{b=0}^{q}\bigoplus_S \SD_S(\oc{m_1}{n_1+b},\oc{m_2}{n_2})\ot \Z\CPart_c(q,\textstyle{\coprod}_b S^1 \cup \del_IS) $$
where $S$ runs over the components of $\SD(\oc{m_1}{n_1+b},\oc{m_2}{n_2})$, 
and $\del_IS$ denotes the interval components of the free boundary of $S$. 
The differential is the differential in $\SD$.  On morphisms, $\Ass$ acts via its action on the second factor.

\begin{lem}\label{SDK} There is a  quasi-isomorphism
of functors $$i^*\SD(\oc{m_1+-}{n_1},\oc{m_2}{n_2})\simeq j^*(\SD\!\ot\! K)(\oc{m_1}{n_1},\oc{m_2}{n_2}):(\Ai)^{op} \rar \Comp.$$
\end{lem}

\begin{proof}
The map 
$$\beta_q: \SD(\oc{m_1+q}{n_1},\oc{m_2}{n_2})\rar \bigoplus_{b=0}^{q} \bigoplus_S\SD_S(\oc{m_1}{n_1+b},\oc{m_2}{n_2}) \ot \CPart_c(q,\textstyle{\coprod}_b S^1 \cup \del_IS).$$
is defined as in the case of $\OC$ by 
taking a Sullivan diagram $G$ with $n_2$ white vertices and $q+1$ leaves $p_0,\dots,p_q$ and $m_1+m_2+n_1$ leaves $\la_1,\dots,\la_{m_1+m_2+n_1}$ to $G'\ot P$ 
where $G'$ is obtained from $G$ by forgetting the leaves $p_0,\dots,p_q$ except for $p_{i_1},\dots,p_{i_b}$, the first occurring leaf in each boundary cycle of $G$ 
with only leaves labeled by 
$p_i$'s, provided that $G$ and $G'$ have the same degree. Otherwise we set $G'=0$.  
The second component, $P$, remembers the configuration defined by the leaves $p_i$ in $\del S$. 
The map $\be_q$ respects the equivalence relation defining Sullivan diagrams and is a chain map as it is a quotient of the same map for $\OC$. 
We are left to check that it is a quasi-isomorphism. 

In the present case, there is  a right inverse $\ga_q$ defined by taking $G\ot P$ to the graph obtained from $G$ by adding the forgotten leaves recorded in $P$ as a trivalent tree, 
in the ordering prescribed by $P$ (as in Figure~\ref{betaga}).
\begin{figure}[h]
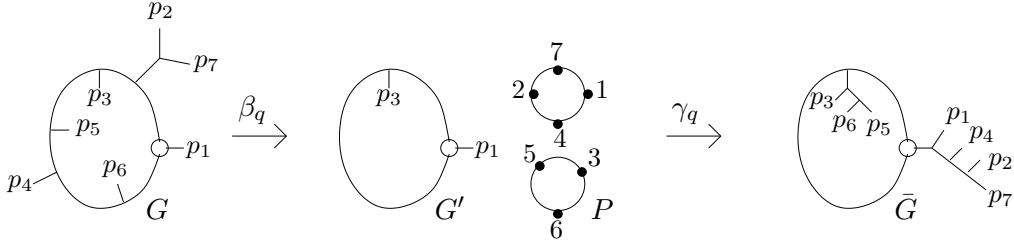

\begin{lpic}{betaga(0.5,0.5)}
\lbl[b]{58,25;$\be_q$}
\lbl[b]{172,25;$\ga_q$}
\lbl[b]{33,0;$G$}
\lbl[b]{110,0;$G'$}
\lbl[b]{150,0;$P$}
\lbl[b]{230,0;$\bar G$}
\lbl[b]{44,16;{\small $p_1$}}
\lbl[b]{34,53;{\small $p_2$}}
\lbl[b]{18,30;{\small $p_3$}}
\lbl[b]{-3,7;{\small $p_4$}}
\lbl[b]{15,22;{\small $p_5$}}
\lbl[b]{22,11;{\small $p_6$}}
\lbl[b]{46,39;{\small $p_7$}}
\lbl[b]{120,16;{\small $p_1$}}
\lbl[b]{94,30;{\small $p_3$}}
\lbl[b]{150,32;{\small $1$}}
\lbl[b]{139,19;{\small $4$}}
\lbl[b]{128,32;{\small $2$}}
\lbl[b]{138,42;{\small $7$}}
\lbl[b]{148,14;{\small $3$}}
\lbl[b]{131,15;{\small $5$}}
\lbl[b]{138,-5;{\small $6$}}
\lbl[b]{244,25;{\small $p_1$}}
\lbl[b]{255,13;{\small $p_2$}}
\lbl[b]{208,29;{\small $p_3$}}
\lbl[b]{250,20;{\small $p_4$}}
\lbl[b]{223,22;{\small $p_5$}}
\lbl[b]{214,23;{\small $p_6$}}
\lbl[b]{255,3;{\small $p_7$}}
\end{lpic}
\caption{The maps $\be_q$ and $\ga_q$ for Sullivan diagrams}\label{betaga}
\end{figure}
This is a natural transformation as the coHochschild and cosimplicial differentials agree under this map using the equivalence relation of Sullivan diagrams.

We have $\be_q\circ \ga_q=id$, so we are left to show that $\ga_q\circ \be_q\simeq id$. The failure of $\ga_q\circ \be_q$ to be the identity is on
graphs with leaves that are ``at or past some white vertices'' of 
their place after applying $\ga_q\circ \be_q$. The chain homotopy is defined by summing over the graphs with each leaf, one at a time, at the white
vertex which has to be past, for each white vertex. This defines a degree 1 map $s:\SD(\oc{m_1+q}{n_1},\oc{m_2}{n_2})\to
\SD(\oc{m_1+q}{n_1},\oc{m_2}{n_2})$ and we need to check that $sd+ds=id-\ga_q\circ \be_q$.  In fact, it is enough to check that we can pass leaves one
at a time: (we omit the sign calculation) for a graph $G$
with a leaf $\la_1$ attached to a trivalent black vertex separated by a single edge from a white vertex, define $s_1(G)$ to be the graph obtained from
$G$ by moving $\la$ to the white vertex. If $G$ is not equivalent to a graph of this form, then define $s_1(G)=0$. 
Then we have  $d(s_1(G))+s_1(d(G))=G'-G$ where $G'$ is obtained from $G$ by moving $\la_1$ to the next half-edge
attached at the white vertex. Indeed, the graphs $G$ and $G'$ are the ``new''  non-zero terms in $d(s_1(G))$.  
(In the example in Figure~\ref{betaga}, the homotopy is a sum of 3 graphs, with $p_7$, then $p_2$,
then $p_4$ attached in turn directly at the white vertex.) 
\end{proof}

\begin{proof}[Proof of Theorem~\ref{H0OC}]
The proof of the theorem then follows directly using the same argument as for Theorem~\ref{OC}: the analogue of Lemma~\ref{OOn1} holds, with
Lemma~\ref{SDK} replacing Lemma~\ref{OCK}, and this last result can then be used repeatedly to prove the theorem. 
\end{proof}



\section{Examples of non-trivial higher string topology operations}\label{non-trivial}

We give in this section, as an example, two infinite families of cycles of increasing degree in the chain complex of Sullivan diagrams, associated
to surfaces of increasing genus (resp.~increasing number of boundary components), which induce a non-trivial action on the Hochschild homology of the 
cohomology algebra $H^*(S^n)$ for any $n\ge 2$, also over the rationals. As explained in the introduction, these operations can be used to define
rational string topology operations for simply-connected manifolds using \cite{lambrechts_stanley}
and \cite{jones} (see also \cite{felix_thomas} or \cite[Sec 6.6]{WahWes08}), which are non-trivial as they act non-trivially on 
$HH_*(H^*(S^n),H^*(S^n))\cong H^*(LS^n)$.

Once the classes are constructed, it will be easy to check that they represent cycles in the chain complex of Sullivan diagrams. The fact that they
act non-trivially on the Hochschild homology of a particular algebra will then imply that they
represent non-trivial homology classes. Very little is known about the homology of the chain complex of Sullivan diagrams. The method presented here 
can be used to detect additional families of classes, and gives us a first insight in the homology of this chain complex.

\medskip

For each $g\ge 1$, we construct two black and white graphs $\mu_g$ and $t_g$ as follows:  
The graph $\mu_g$ is obtained from the graph $l_{2g+2}$, a single white vertex with $2g+2$ leaves, 
by adding a tree connecting the odd entries $1,3,\dots,2g+1$ together and labeling the remaining
leaves $1,2,\dots,g+1$ in their cyclic order. (See Figure~\ref{tgmg} for a picture of $\mu_1,\mu_2$ and $\mu_3$ both as the black and white graphs just
described, and in the more classical picture of Sullivan chord diagrams using the isomorphism given in \cite[Thm 5.9]{WahWes08}.)
\begin{figure}[h]
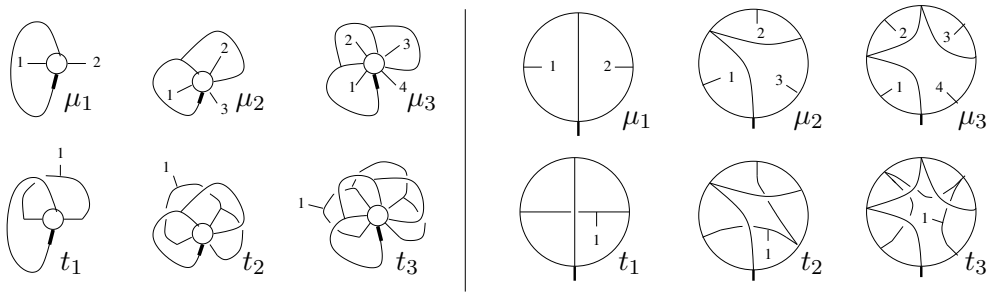

\begin{lpic}{tgmg(0.48,0.48)}
\lbl[b]{20,50;$\mu_1$}
\lbl[b]{67,50;$\mu_2$}
\lbl[b]{114,50;$\mu_3$}
\lbl[b]{18,5;$t_1$}
\lbl[b]{68,5;$t_2$}
\lbl[b]{110,5;$t_3$}
\lbl[b]{173,45;$\mu_1$}
\lbl[b]{220,45;$\mu_2$}
\lbl[b]{265,45;$\mu_3$}
\lbl[b]{171,5;$t_1$}
\lbl[b]{221,5;$t_2$}
\lbl[b]{266,5;$t_3$}
\end{lpic}
\caption{The graphs $t_g$ and $\mu_g$ for $g=1,2,3$ as black and white graphs and as classical Sullivan diagrams}\label{tgmg}
\end{figure}

The graph $t_g$ is then constructed from $\mu_g$ by adding a second tree connecting the remaining leaves at the white vertex, with a
leaf on the new tree as shown in the figure for $t_1,t_2$ and $t_3$. 

The graphs $\mu_g$ and $t_g$ both have degree $2g+1$. 
It is easy to check that they have boundary 0: they have an even number of faces, and each face $d_{2i}$ cancels with the face
$d_{2i+1}$. Hence the graphs are cycle. The computation given below of a non-trivial action in homology using these graphs implies that they
represent non-trivial homology classes. 

We consider $\mu_g$ as living in $\SD(\oc{0}{g+1},\oc{0}{1})$, that is each leaf is considered as marking an incoming closed boundary component. The
topological type of the associated cobordism, which is the surface obtained by thickening the graph (with the white vertex defining a boundary
component), is a surface of genus 0 with $g+2$ boundary components, $g+1$ of which being incoming and the last one being outgoing. 

The graph $t_g$ is considered as living in $\SD(\oc{0}{1},\oc{0}{1})$. The surface obtained by thickening the has  two boundary components, the
incoming and the outgoing one, and an Euler characteristic computation shows that it has genus $g$. 
In the category of Sullivan diagrams, $t_g$ is actually the precomposition of $\mu_g$ with a representative for the $g+1$--legged coproduct of degree 0 (the dual of the $g$--fold Chas-Sullivan product).

\medskip

Consider the Frobenius algebra $A=\Z[x]/(x^2)$ concentrated in degree 0, with the coproduct given by  $\nu(1)=1\ot x + x\ot 1$ and $\nu(x)=x\ot x$.
This is an $H_0(\OO)$--algebra. (Up to degree shift and signs, this is the cohomology algebra of $S^n$.) 
The reduced Hochschild complex of $A$ has two generating class in each degree, namely $1\ot x\ot\cdots\ot x$ and $x\ot x\ot\cdots\ot x$, and the
differential is 0 except on the terms $1\ot x\ot\cdots \ot x$ in degree $2i$, which have differential $2(x\ot\cdots\ot x)$ in degree $(2i-1)$. In
particular, the classes $1\ot x\ot\cdots \ot x$ in odd degree are cycles representing non-trivial homology classes in $HH_*(A,A)$, and so are the classes 
$x\ot\cdots \ot x$, though those living in odd degrees in this case only generate a $\Z/2$ in homology. 

\begin{prop}\label{tprop}
Let $A=\Z[x]/(x^2)$  be the $H_0(\OO)$--algebra defined above and let 
$$F: \OC(\oc{0}{1},\oc{0}{1})\ot C_*(A,A)\rar C_*(A,A)$$
be the action of Sullivan diagrams of Theorem~\ref{H0OC} on the Hochschild complex of $A$. Let $t_g\in\SD(\oc{0}{1},\oc{0}{1})$,
and $x\in C_0(A,A)$ be as above. 
Then $$F(t_g\ot x)=1\ot x\ot \cdots\ot x\in C_{2g+1}(A,A).$$ In particular, 
$H_*(F(t_g\ot -)):HH_0(A,A)\to HH_{2g+1}(A,A)$ is non-zero, also with rational coefficients. 
\end{prop}

The following corollary follows from the fact that $t_g$ can be expressed in Sullivan diagrams as the $g$--fold composition of the coproduct
corresponding to the dual of the Chas-Sullivan product, which takes  $x\in C_0(A,A)$ to $x\ot\cdots\ot x\in C_0(A,A)^{\ot g+1}$, followed by $\mu_g$,  
though the result can  also be checked directly.

\begin{cor}
The class $\mu_g\in\SD(\oc{0}{g+1},\oc{0}{1})$ induces  a non-trivial map 
 $HH_0(A,A)^{\ot g+1}\to  HH_{2g+1}(A,A)$ for $A=\Z[x]/(x^2)$ as in the proposition. 
\end{cor}

\begin{proof}[Proof of Proposition~\ref{tprop}]
Indeed, following the recipe given in \cite[6.2]{WahWes08} for reading of the action of a graph on an algebra, we put $x$ at the leaf of $t_g$ and
read off the graph minus the white vertex as a composition of operations in $A$. In the case of $t_g$,  this graph is the union of the two trees used
to construct the graph above. 
The first tree, as an operation on $A$, is the iterated coproduct $\nu^q(1)$, and the second tree is 
an iterated coproduct $\nu^g(x)=x\ot\cdots\ot x$.  We have 
$\nu^g(1)=1\ot x\ot\cdots\ot x + x\ot 1\ot x\ot \cdots\ot x + \dots + x\ot\cdots\ot x\ot 1$, but only the first term gives a non-zero result in the
reduced Hochschild complex. The resulting Hochschild chain is an intertwine of the two iterated coproduct, as prescribed by the way the trees are
attached to the white vertex. Explicitly, it is the class  $1\ot x\ot \cdots \ot x$ with $2g+1$ $x$'s, as announced. 

For the statement about the action in homology, we just note that $x$ and  $1\ot x\ot \cdots \ot x$ both  are cycles representing non-trivial homology
class, each generating a free $\Z$--summand. 
\end{proof}


We consider briefly the case of the 
cohomology algebra $H^*(S^n)$, which is a degree shifted version of the algebra $A$ above. Just as $A$ 
it is generated by two classes: $1\in H^0(S^n)$ and $x\in H^n(S^n)$, with $1$ a unit for the product, and
$x.x=0$. The coproduct $\nu$ has degree $n$
and satisfies $\nu(1)=1\ot x \pm x\ot 1$ and $\nu(x)=x\ot x$. (The sign does not play a role here.) 
Again just as above, the reduced Hochschild complex of $H^*(S^n)$ is generated by a class $1\ot x\ot \cdots\ot x$, though now  in total degree $in-i$ for each $i\ge 0$, and 
$x\ot \cdots\ot x$, now in total degree $in-i+1$ for each $i\ge 1$. When $n$ is odd, the differential of the complex is identically 0, and when $n$ is
even, it is 0 except in degrees $2in-2i$, where it is multiplication by 2. In particular, the collection 
$$1,1\ot x, x, (1\ot x\ot x), x\ot x, 1\ot x\ot x\ot x, x\ot x\ot x, (1\ot x\ot x\ot x\ot x), x\ot x\ot x \ot x,\dots$$
ordered in increasing degree, represents a collection of generator of the Hochschild homology of $H^*(S^n)$, 
where the generators in bracket are only to be used for $n$ is odd. 
Under the isomorphism $HH_{-*}(H^*(S^n),H^*(S^n))\cong H^*(LM))$ for $n\ge 2$,  the class $[x]\in HH_0(H^*(S^n),H^*(S^n))$ is dual to  the
$n$-dimensional class of constant loops coming from the inclusion $S^n\inc LS^n$. 

The graded vector space $H^*(S^n)$ is a dimension $n$ Frobenius algebra in the sense of \cite[Sec 6.4]{WahWes08}, and the action of Sullivan diagrams
is replaced by a degree shifted version of the above action. The calculation goes through in the exact same way and there are no sign issues as signs
did not play a role in the computation. 
Hence $t_g$ and $\mu_g$ define non-trivial string topology operations.


\section{Appendix A: The chain complex associated to a cosimplicial set}\label{cosimpsec}

In this section, we show that the chain complex associated to a cosimplicial set has homology concentrated in degree 0, generated by the elements in
cosimplicial degree 0 on which  $d_0$ and $d_1$ agree.  
This result was suggested to us by Bill Dwyer. It can be found in alternative form in \cite[Prop 23.10]{May67},
\cite[2.2-2.4]{Bou87}, \cite[Lem 4.1]{Goo98}. 

\medskip

We denote by $\De$ the category of finite sets $[n]=\{0,\dots,n\}$, $n\ge 0$,  and monotone maps. A cosimplicial set is a functor $X^\bullet:\De\to Sets$. 
Recall that the coface map $d^i:X^q\to X^{q+1}$ is the image under $X^\bullet$ of the injection, also denote  $d^i:[q]\inc [q+1]$, which omits $i$, and the 
codegeneracy $s^j:X^{q}\to X^{q-1}$ is the image of the surjection $s^j:[q]\surj [q-1]$ mapping $j$ and $j+1$ to $j$. 
Any injection in $\De$ is a composition of cofaces and any surjection is a composition of codegeneracies.  

\begin{lem}\label{cosimplem}
Let $X^\bullet$  be a cosimplicial set, and $x\in X^q$ a $q$--simplex for some $q\ge 0$. 
There exists a unique $y\in \cup_{p\le q} X^p$ so that $x=D(y)$ for $D=d^{i_1}\dots\, d^{i_{q-p}}\in \De(p,q)$ a composition of
coface maps (i.e.~an injection) and $y$  not itself a coface. Moreover, the injection $D$ is unique as a morphism in $\De$ 
unless $y\in X^0$ and $d^0y=d^1 y$. In that last case,
we have $Dy=D'y$ for any injection $D,D'\in \De(0,q)$. 
\end{lem}

\begin{proof}
Suppose $x=D_1y_1=D_2y_2$ for $D_1,D_2$ compositions of cofaces and $y_1\in X^{p_1}$, $y_2\in X^{p_2}$ which are not themselves cofaces. Let $S_1$ be a composition of
codegeneracies (i.e.~a surjection) such that $S_1D_1=id\in \De(p_1,p_1)$. Then $y_1=S_1D_1y_1=S_1D_2y_2$. Now write $S_1D_2=D_3S_3\in \De(p_2,p_1)$ 
for some $D_3,S_3$ compositions of
cofaces and codegeneracies respectively---this is possible as any map in $\Delta$ can be written as a composition of a surjection followed by an
injection. As $y_1$ is not a coface, we
must have $D_3=id$ and $y_1=S_3y_2$. By symmetry, we also have that $y_2=S'_3y_1$ for some composition of codegeneracies $S_3'$. But this can only 
happen if $p_1=p_2$, and $S_3=S_3'=id$. Hence  $y_1=y_2$. 

For the uniqueness of $D$, suppose first that $y$ is not in cosimplicial degree 0. As we had $S_1D_1=D_3S_3$ in the above computation, with $S_1$ any
left inverse of $D_1$, we can conclude from the above that $S_1D_2=id$ for any left inverse $S_1$
of $D_1$  in the category $\Delta$. But if $D_1\neq D_2$, the corresponding maps in $\Delta$ are different, and if the
source of $D_1,D_2$ is not 0, then there exists an $S_1$ such that $S_1D_1=id$ in $\Delta$ but $S_1D_2\neq id$. 

We are left with the case $y$ in degree 0. (Here we have that $S_1D_1=id$ for any $D_1:[0]\inc [q]$ and the unique surjection $S_1:[q]\surj[0]$.) 
Suppose first that $d^0y=d^1y$. We then prove by induction that $Dy=d^0\dots\, d^0y$ for any composition of coface maps. Indeed, suppose
it is true for compositions of length $k-1$. Then write $D$ of length $k$ as $D=d^id^0\dots\, d^0$. As $d^id^0=d^0d^{i-1}$ for any $i>0$,
we have that $Dy=d^id^0\dots\, d^0y=d^0d^{i-1}d^0\dots\, d^0y=d^0\dots\, d^0y$. 

Finally, if $d^0y\neq d^1y$, we can distinguish different compositions $D_1, D_2$ using the following fact: if $D_1,D_2:[0]\to[q]$ are compositions
of coface maps in $\Delta$ with $D_1(0)>D_2(0)$, 
there exists a composition of degeneracies $S:[q]\to[2]$ such that $SD_1=d_0$ and $SD_2=d_1$. Then we must have $D_1(y)\neq D_2(y)$ as 
for such a choice of $S$ we have $SD_1(y)\neq SD_2(y)$.  
\end{proof}

By a {\em semi-cosimplicial set}, we mean a functor $Y^\bullet:\De^{inj}\to Sets$, where $\De^{inj}$ is the subcategory of $\De$ with the same objects
but with only monotone injections as morphisms. So a semi-cosimplicial set has coface maps satisfying the usual relations, but no codegeneracies. 
Let $\Inj(r)^\bullet$ be the semi-cosimplicial set represented by $r$, i.e.~with $\Inj(r)^q=\De^{inj}(r,q)$, and with the coface maps
defined by composition.  It is the free semi-cosimplicial set generated by one element in cosimplicial degree $r$. 

\smallskip

The above lemma associates to every simplex $x$ in a cosimplicial set $X^\bullet$, a unique minimal $y$ such that $x=Dy$. Moreover, the lemma shows
that, if $d^0y\neq d^1y$ then $Dy\neq D'y$ for any injection $D\neq D'$ in $\De$, i.e.~$y$ generates a free semi-cosimplicial set in $X^\bullet$, 
and if $d_0y=d_1y$, then $Dy=D'y$ for all $D,D'$ of the same
length, i.e.~$y$ generates a constant sub-cosimplicial set (with $y$ necessarily in degree 0). 
This proves the following:

\begin{cor}\label{splitcor}
Any cosimplicial set $X^\bullet$ splits as a semi-cosimplicial set as the disjoint union of constant subcosimplicial sets, one for each $y\in X^0$ such
that $d^0y=d^1y$,  and free semi-cosimplicial sets $\Inj(r)^\bullet$ for various $r$'s. 
\end{cor}

To a (semi)-cosimplicial set $X^\bullet$, we can associate a chain complex $\Z X:=\oplus_q \Z X^q$, with $\Z X^q$ in degree $-q$, and with differential 
$d_X=\sum (-1)^id^i$, the alternating sum of the coface maps.

\begin{prop}\label{cosimpprop}
For any $r\ge 0$, the chain complex $(\Z \Inj(r),d_{Inj})$ associated to the free semi-cosimplicial set $\Inj(r)^\bullet$ has no homology. 
\end{prop}

\begin{proof}
In degree $q$, $\Z \Inj(r)$ is the free module on all injection $[r]\inc [q]$, and the boundary map $d:\Z \Inj(r)^q\to \Z \Inj(r)^{q+1}$ is the alternating
sum of the coface maps: $d=\sum_{i=1}^{q+1}(-1)^id^i$, where $d^i([r]\sta{x}{\rar}[q])$ is the composition $[r]\sta{x}{\rar} [q]\sta{d_i}{\rar} [q+1]$ where the
second map does not hit the $i$th element. 

Define a chain homotopy $s:\Z \Inj(r)^q\to \Z \Inj(r)^{q-1}$ by  $s([r]\sta{x}{\rar}[q])$ is $(-1)^qx$, considered as landing in $[q-1]$ if $q$ is not
in the image of $x$, and is 0 otherwise. 

Then if $q$ is in the image of $x$, we have $sd+ds([r]\sta{x}{\rar} [q])=s(\sum_{i=1}^{q+1} (-1)^i([r]\sta{x}{\rar}[q]\sta{d_i}{\rar}[q+1])+0$
equals $(-1)^{q+1}(-1)^{q+1}(x)$ as only the last term in the sum will not hit $q+1$. On the other hand, if $q$ is not in the image of
  $x$, then  
\begin{align*}sd+ds([r]\sta{x}{\rar} [q])\ =\ &s\Big(\sum_{i=1}^{q+1} (-1)^i([r]\sta{x}{\rar}[q]\sta{d_i}{\rar}[q+1])\Big)\\
&+(-1)^q\sum_{i=1}^{q} (-1)^i([r]\sta{sx}{\rar}[q-1]\sta{d_i}{\rar}[q])\end{align*}
As $x$ does not hit $q$, none of the terms in the first sum will hit $q+1$ and, given that $s$ in this case introduces a sign
$(-1)^{q+1}$, all but the last term will cancel with the terms in the second sum. The only term that does not cancel is equal to $x$.  
\end{proof}

We can now prove the result we use in Section~\ref{compsec}:

\begin{prop}\label{cosimpcor}
Let $X^\bullet$ be a cosimplicial set, and $C(X,M):=\oplus_q\Z X^q\ot M$ the associated chain complex with trivial coefficients in a
module $M$ and differential $d_X=\sum (-1)^id^i$.  Then $$H_*(C(X,M))=\bigoplus_{\begin{subarray}{c}x\in X^0\\ d^0x=d^1 x\end{subarray}}M$$ 
with each copy of $M$ concentrated in degree 0.  
\end{prop}

\begin{proof}
By Corollary~\ref{splitcor} the cosimplicial set $X^\bullet$ splits as a semi-cosimplicial set into a disjoint union of free semi-cosimplicial sets
$\Inj(r)$ and one constant cosimplicial set for each $x\in X^0$ satisfying $d^0x=d^1x$. 
As the differential in $C(X,M)$ only uses the semi-cosimplicial set structure, we also get a splitting of $C(X,M)$ into
such direct summands. By  Proposition~\ref{cosimpprop}, the free semi-cosimplicial sets do not contribute anything to the homology. On the other hand, the chain
complex associated to a constant cosimplicial set is the cochain complex of a point, and hence contributes a copy of the module $M$ in degree 0. 
\end{proof}

\section{Appendix B: Graph complexes}\label{GraphApp}

In this appendix, we briefly recall the necessary ingredients to define a chain complex of graphs as used in the paper. We exemplify how this allows to replace certain algebraic equations by graph equations. 

\medskip

A {\em graph} is formally defined as a tuple $G=(V,H,s,i)$ where $V$ is its set of {\em vertices}, $H$ its set of {\em half-edges}, $s:H\to V$ is the {\em source map}, and $i:H\to H$ is an involution. Unless otherwise stated, graphs will be assumed to be at least 3-valent, i.e.~such that $|v|:=|s^{-1}(v)|\ge 3$ for all $v\in V$.  An {\em edge} in the graph is a pair of half-edges $(h,i(h))$, with $i(h)\neq h$. A {\em leaf} is a fixed point of the involution: $h=i(h)$. 
We allow also the exceptional graph made out of a single leaf and no vertices, for example to describe identity morphisms in e.g.~the category $\Ai$.
An {\em orientation} for a graph $G$ is a unit vector in $\det(\RR(V\sqcup H))$.

To a graph $G$, one can associate a 1-dimensional CW-complex, its {\em realization}, by attaching edges and leaves along their end vertices, in the way prescribed by the source map. A {\em tree} is a graph such that its realization is contractible. 
A {\em planar tree} is a tree, together with the additional data of an embedding of the realization of the tree into the plane.  

A {\em fat graph} is a graph equipped with a cyclic ordering of the half-edges $s^{-1}(v)$ at each vertex $v$. Planar graphs are examples of fat graphs, the embedding in the plane giving such a cyclic ordering at each vertex.

Given a graph $G=(V,H,s,i)$ and an edge $e=(h_1,h_2)$ of $G$ with $s(h_1)\neq s(h_2)$, one can define the collapse $G/e$ of $G$ along $e$ by removing the two half-edges $h_1,h_2$ and identifying their source vertices. If $G$ had a fat structure, $G/e$ inherits a fat structure. If $G$ is oriented,  $G/e$ inherits an orientation as follows: writing the orientation of $G$ as $v_1\w v_2\w h_1\w h_2\w X$ for $v_1=s(h_1)$ and $v_2=s(h_2)$, we defined the orientation of $G/e$ to be $v\w X$ where $v$ denotes the vertex obtained by identifying $v_1$ and $v_2$.  
Now a {\em blow-up} of an (oriented, fat) graph $G$ is an (oriented, fat) graph $\hat G$ such that there exists an edge $e$ in $\hat G$ with $\hat G/e\cong G$ as (oriented, fat) graphs.

Blow-ups can be used to define a chain complex of oriented fat graphs as follows. 
First we define the {\em degree} of a graph $G$ as $deg(G)=\sum_{v\in V}|v|-3$. Now 
one defines a chain complex which, in degree $d$, is the free $\Z$--module on the set of oriented degree $d$ fat graphs, modulo the relation that $(-1)$ reverses the orientation. The {\em differential} is defined as the sum of all possible blow-ups:  $$d(G)=\sum_{\begin{array}{l}(\hat G,e)\\\hat G/e\cong G\end{array}}\hat G$$
(See Figure~\ref{m3} for an example.)

\medskip

Graph complexes are used in the paper to describe $\Ai$--algebras, compute their Hochschild complex and compute operations on the Hochschild complex. 
We give here an example on how to go between graphs and algebra, and how graph orientations translate to signs. 
Hopefully it illustrates that, in an appropriate context, oriented graphs can be easier to work with than algebraic formulas with signs.

\begin{ex}[The map $m_2\in \Ai(2,1)$ is homotopy associative.]\label{m3ex}
{\rm 
Let $m_2$ be the graph with one vertex $v$ and three leaves $h_0,h_1,h_2$ attached to it. (So $s(h_i)=v$ and $i(h_i)=h_i$ for all $i$.
) We give $m_2$ the ``left-right top-down'' orientation $h_1\w h_2\w v\w h_0$, and we interpret $m_2$ as a morphism from $2$ to $1$ in $\Ai$ by labeling $h_1$ and $h_2$ as the first and second incoming leaves, and $h_0$ as the root/outgoing leave. Similarly, let $m_3$ be the graph with a single vertex $w$ and leaves $l_0,l_1,l_2,l_3$ attached to it, with orientation $l_1\w l_2\w l_3\w w\w l_0$. 

The graph $m_3$ has degree 1. Its differential has two terms, shown in Figure~\ref{m3}, as there are two ways of blowing $w$ up. 
Both blow-ups can be described as having vertex set $V=\{w_1,w_2\}$, half-edges $\{l_0,l_1,l_2,l_3,e_1,e_2\}$ with $i(e_1)=e_2$, $s(e_i)=w_i$, though with different source maps for the half-edges $l_i$, as described by the picture. To compute their orientation, we first write  
$l_1\w l_2\w l_3\w w\w l_0=-w\w l_1\w l_2\w l_3\w l_0$. Now 
by definition of the differential, both terms come with the orientation $-w_1\w w_2\w e_1\w e_2\w  l_1\w l_2\w l_3\w l_0$. 

The oriented graph $m_3$ defines an element in $\Ai(3,1)$ with $l_1,l_2,l_3$ the three incoming leaves and $l_0$ the outgoing leaf. 
We would like to compare the terms in its differential to the compositions $m_2\circ(m_2 + id)$ and $m_2\circ(id + m_2)$. As graphs, these are indeed to two terms, but the orientation of $m_2\circ(m_2 + id)$ is $(l_1\w l_2\w w_1 \w e_1) \w (e_2\w l_3\w w_2\w l_0)$  and that of 
 $m_2\circ(id + m_2)$ is $(l_2\w l_3\w w_1\w e_1) \w (l_1\w e_2\w w_2\w l_0)$, the orientation of a composition being the juxtaposition of the orientations. One then checks that, with the above choices of oriented representatives for $m_2,m_3$, we have $d(m_3)=m_2\circ(m_2 + id)-m_2\circ(id + m_2)$.
}\end{ex}

\begin{figure}[h]
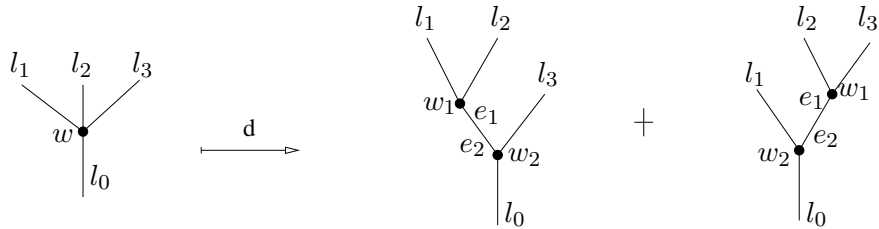

\begin{lpic}{m3(0.62,0.62)}
\lbl[b]{9,18;$w$}
\lbl[b]{0,32;$l_1$}
\lbl[b]{13,32;$l_2$}
\lbl[b]{26,32;$l_3$}
\lbl[b]{17,8;$l_0$}
\lbl[b]{89,24;$w_1$}
\lbl[b]{107,13;$w_2$}
\lbl[b]{86,42;$l_1$}
\lbl[b]{102,42;$l_2$}
\lbl[b]{112,30;$l_3$}
\lbl[b]{105,0;$l_0$}
\lbl[b]{99,22;$e_1$}
\lbl[b]{96,15;$e_2$}
\lbl[b]{177,27;$w_1$}
\lbl[b]{160,13;$w_2$}
\lbl[b]{156,30;$l_1$}
\lbl[b]{167,42;$l_2$}
\lbl[b]{180,42;$l_3$}
\lbl[b]{169,0;$l_0$}
\lbl[b]{168,25;$e_1$}
\lbl[b]{171,17;$e_2$}
\end{lpic}
\caption{The differential of $m_3$}\label{m3}
\end{figure}

More generally, the left-right top-bottom orientation gives a ``canonical'' orientation for planar  rooted trees (or trees with a leaf specified as output) and this can be used to translate any formula given in terms of oriented graphs in $\Ai$ as an algebraic formula in terms of $m_i$'s. Likewise, one can choose a canonical orientation for the graph $l_k$ of Section~\ref{sec1} and translate graph equations for the Hochschild complex of a dg-algebra to algebraic equations.   

\medskip 

We refer the reader to \cite{WahWes08}, in particular Sections 2,3 and 7 for more details and further examples of graph complexes modeling algebraic structures. 

\bibliography{biblio}

\end{document}